\documentclass[12pt]{amsart}

\usepackage{amssymb}
\usepackage{array}

\theoremstyle{plain} 

\newtheorem{theorem}{Theorem}[section]
\newtheorem{lemma}[theorem]{Lemma}
\newtheorem{remark}[theorem]{Remark}

\newtheorem{corollary}[theorem]{Corollary}
\newtheorem{conjecture}[theorem]{Conjecture}
\newtheorem{definition}[theorem]{Definition}
\newtheorem{proposition}[theorem]{Proposition}

\newtheorem{condition}[theorem]{Condition}

\newcommand{\makeinvisible}[1]{}

\numberwithin{equation}{section}

\newcommand{\cc}{{\mathbb C}}
\newcommand{\pp}{{\mathbb P}}
\newcommand{\rr}{{\mathbb R}}

\newcommand{\gggg}{{\mathbb G}}

\newcommand{\aaa}{{\mathbb A}}
\newcommand{\Gm}{\gggg _m}

\newcommand{\Uu}{{\mathcal U}}

\newcommand{\stackmod}{{/\!\! /}}

\newcommand{\dualdel}{{{\mathbb D} \partial}}

\begin{document}

\author[C. Simpson]{Carlos Simpson}
\address{\newline CNRS, Laboratoire JAD, UMR 7351
\newline Universit\'e Nice Sophia Antipolis
\newline
06108 Nice Cedex 2\newline France}

\title[Dual boundary complex]{The dual boundary complex of the $SL_2$ character variety of 
a punctured sphere}



\begin{abstract}
Suppose $C_1,\ldots , C_k$ are generic conjugacy classes in $SL_2(\cc )$.
Consider the character variety of local systems on $\pp^1-\{ y_1,\ldots , y_k\}$
whose monodromy transformations around the punctures $y_i$ are in the respective conjugacy
classes $C_i$. We show that the dual boundary complex of this character variety
is homotopy equivalent to a sphere of dimension $2(k-3)-1$. 
\end{abstract}


\maketitle

\section{Introduction}
\label{sec-intro}

Given a smooth projective variety $X$, choose a normal crossings compactification 
$\overline{X}=X\cup D$ and define a simplicial set called the {\em dual boundary complex}
$\dualdel X$, containing the combinatorial information about multiple intersections of
divisor components of $D$. Danilov, Stepanov and Thuillier have shown 
that the homotopy type of $\dualdel X$ is
independent of the choice of compactification, and this structure has been the subject of
much study. 

We consider the case when $X={\rm M}_B(S; C_1,\ldots , C_k)$ is the character variety, of 
local systems on a punctured sphere $S\sim \pp ^1-\{ y_1,\ldots , y_k\}$ such that 
the conjugacy classes of the monodromies around the punctures are given by $C_1,\ldots , C_k$
respectively \cite{Letellier}. If these conjugacy classes satisfy a natural genericity condition 
then the character variety is a smooth affine variety. We prove that its
dual boundary complex is a sphere of the appropriate dimension (see Conjecture \ref{geopw}), 
for local systems of rank $2$.

\begin{theorem}
\label{main}
Suppose $C_1,\ldots , C_k$ are conjugacy classes in $SL_2(\cc )$
satisfying the genericity Condition \ref{verygen}. Then the dual boundary complex
of the character variety is homotopy equivalent to a sphere:
$$
\dualdel {\rm M}_B(S; C_1,\ldots , C_k) \sim S^{2(k-3)-1}.
$$
\end{theorem}

This statement is a part of a general conjecture about the boundaries of moduli spaces of
local systems \cite{KNPS}. The conjecture says that the
dual boundary complex of the character variety or ``Betti moduli space'' should be a
sphere, and that it should furthermore be naturally identified with the sphere at infinity
in the ``Dolbeault'' or Hitchin moduli space of Higgs bundles. 
We will discuss this topic in further detail in Section \ref{rel-hitch} at the
end of the paper. 

The case $k=4$ of our theorem is a consequence of the
Fricke-Klein expression for the character variety, which was indeed the motivation for the
conjecture. The case $k=5$ of Theorem \ref{main} has been proven by Komyo \cite{Komyo}. 

\subsection{Strategy of the proof}
Here is the strategy of our proof. We first notice that it is 
possible to make some reductions, based on the following observation
(Lemma \ref{negligeable}): 
if $Z\subset X$ is a smooth closed subvariety of a smooth quasiprojective variety, such that 
the boundary dual complex is contractible $\dualdel Z\sim \ast$, 
then the natural map $\dualdel X\rightarrow \dualdel (X-Z)$ is a homotopy
equivalence. This allows us to remove some subvarieties which will be ``negligeable'' for the
dual boundary complex. The main criterion is that
if $Z= \aaa^1\times Y$ then $\dualdel Z\sim \ast$
(Corollary \ref{affine}). Together, these two statements allow us successively to remove
a whole sequence of subvarieties (Proposition \ref{decomp}). 

The main technique is to express the moduli space 
${\rm M}_B(S; C_1,\ldots , C_k)$ in terms of a decomposition of $S$ into a sequence of
``pairs of pants'' $S_i$ which are three-holed spheres. 
The decomposition is obtained by cutting $S$ along $(k-3)$ circles denoted $\rho _i$. 
In each $S_i$, there is one
boundary circle corresponding to a loop $\xi _i$ 
around the puncture $y_i$, and two other boundary circles $\rho _{i-1}$ and $\rho _i$ 
along which $S$ was cut. At the start and the end of the sequence,
two of the circles correspond to $\xi _1,\xi _2$ or $\xi _{k-1},\xi _k$ and only one to a cut.
One may say that $\rho _1$ and $\rho _{k-1}$ are confused with the original
boundary circles $\xi _1$ and $\xi _k$ respectively.  

We would like to use this decomposition
to express a local system $V$ on $S$ as being the result of ``glueing'' together
local systems $V|_{S_i}$ on each of the pieces, glueing across the circles $\rho _i$. 
A basic intuition, which one learns from the
elementary theory of classical hypergeometric functions, is that a local system of
rank $2$ on a three-holed sphere is determined by the conjugacy classes of its three
monodromy transformations. This is true generically, but one needs to take some care in
degenerate cases involving potentially irreducible local systems, as will be discussed below.

The conjugacy classes of the monodromy transformations around $\rho _i$ are determined, except in
some special cases, by their traces. The special cases are when the traces are $2$ or $-2$. 

If we assume for the moment the uniqueness of $V|_{S_i}$ as a function of $C_i$ and the
traces $t_{i-1}$ and $t_i$ of  the monodromies around $\rho _{i-1}$ and $\rho _i$ 
respectively, then the local system $V$ is roughly speaking determined by specifying
the values of these traces $t_2,\ldots , t_{k-2}$, plus the glueing parameters.
The glueing parameters should respect the monodromy transformations, and are defined modulo
central scalars, so each parameter is an element of $\Gm$. In this rough picture then,
the moduli space could be viewed as fibering over $(\aaa^1)^{k-3}$ with fibers
$\Gm ^{k-3}$. 

The resulting coordinates are classically known as {\em Fenchel-Nielsen coordinates}.
Originally introduced to parametrize $PGL_2(\rr )$ local systems corresponding to 
points in Teichm\"uller space, they have been extended to the complex character variety
by Tan \cite{Tan}. 

In the above discussion we have taken several shortcuts. We assumed that the traces $t_i$ determined
the  monodromy representations, and in saying that the glueing parameters would
be in $\Gm$ we implicitly assumed that these monodromy transformations were diagonal with
distinct eigenvalues. These conditions correspond to saying $t_i \neq 2,-2$. 

We also assumed that the local system $V|_{S_i}$ was determined by $C_i$,
$t_{i-1}$ and $t_i$. This is not in general true if it can be reducible, which is to say
if there is a non-genericity relation between the conjugacy classes. The locus where that
happens is somewhat difficult to specify explicity since there are several possible choices
of non-genericity relation (the different choices of $\epsilon _i$ in Condition \ref{Kgen}).
We would therefore like a good way of obtaining such a rigidity even over the non-generic
cases. 

Such a property is provided by the notion of {\em stability}. One may envision assigning parabolic
weights to the two eigenvalues of $C_i$ and assigning parabolic weights zero over $\rho _j$.
The parabolic weights induce a notion of {\em stable} local system over $S_i$. But in fact
we don't need
to discuss parabolic weights themselves since the notion of stability can also be defined
directly: a local system $V_i$ on $S_i$ is {\em unstable} if it admits a rank $1$ subsystem
$L$ such that the monodromy matrix in $C_i$ acts on $L$ by $c_i^{-1}$ (a previously chosen
one of the two 
eigenvalues of $C_i$). It is {\em stable} otherwise. Now, it becomes true that a stable
local system is uniquely determined by $C_i$,
$t_{i-1}$ and $t_i$. This will be the basis of our calculations in Section \ref{sec-main},
see Corollary \ref{hyperge}. 

The first phase of our proof is to use the possibility for reductions given by
Proposition \ref{decomp} to reduce to the case of the open subset 
$$
M'\subset {\rm M}_B(S; C_1,\ldots , C_k)
$$
consisting of local systems $V$ such that $t_i\in \aaa^1-\{ 2,-2\}$ and 
such that $V|_{S_i}$ is stable. In order to make these reductions, we show in Sections
\ref{sec-splitting} and \ref{sec-decomp-unstable} that the strata where some $t_i$ is $2$ or $-2$, or
where some 
$V|_{S_i}$ is unstable, have a structure of product with $\aaa^1$, hence by Lemma \ref{affine}
these strata are negligeable in the sense that Lemma \ref{negligeable} applies. 

For the open set $M'$, there is still one more difficulty. The glueing parameters depend {\em a priori}
on all of the traces, so we don't immediately get a decomposition of $M'$ as a product. A calculation
with matrices and a change of coordinates allow us to remedy this and we show in
Theorem \ref{fenchelnielsen} that $M' \cong {\bf Q}^{k-3}$ where ${\bf Q}$ is a space of choices of
a trace $t$ together with a point $[p,q]$ in a copy of $\Gm$. 

It turns out that this family of multiplicative groups over $\aaa^1-\{ 2,-2\}$ is twisted:
the two endpoints of the fibers $\Gm$ get permuted as $t$ goes around $2$ and $-2$. This
twisting property is what makes it so that 
$$
\dualdel {\bf Q} \sim S^1,
$$
and therefore by \cite[Lemma 6.2]{Payne}, 
$\dualdel ({\bf Q}^{k-3})\sim S^{2(k-3)-1}$. This calculates $\dualdel M'$ and hence
also $\dualdel {\rm M}_B(S; C_1,\ldots , C_k)$ to prove Theorem \ref{main}. 

We should consider the open subset $M'$ as the natural domain of definition of the Fenchel-Nielsen
coordinate system, and the components in the
expression $M' \cong {\bf Q}^{k-3}$ are the Fenchel-Nielsen coordinates. 

\subsection{Relation with other work}

What we are doing here is closely related to a number of things. Firstly, as pointed out above,
our calculation relies on the Fenchel-Nielsen coordinate system coming from a pair of pants decomposition,
and this is a well-known technique. Our only contribution is to keep track
of the things which must be removed from the domain of definition, and of the precise form of the
coordinate system, so as to be able to conclude the structure up to homotopy of the dual boundary complex. 

A few references about Fenchel-Nielsen coordinates include 
\cite{FenchelNielsen} \cite{Goldman} \cite{ParkerPlatis} \cite{Wolpert},
and for the complex case Tan's paper \cite{Tan}.
Nekrasov, Rosly and Shatashvili's work on bend parameters 
\cite{NekrasovRoslyShatashvili} involves similar coordinates
and is related to the context of polygon spaces \cite{GodinhoMandini}.
The work of Hollands and Neitzke \cite{HollandsNeitzke} 
gives a comparison between Fenchel-Nielsen and Fock-Goncharov coordinates
within the theory of spectral networks \cite{GMN}. 
Jeffrey and Weitsman \cite{JeffreyWeitsman} consider what is the effect of
a decomposition, in arbitrary genus, on the space of representations into a compact group.  
Recently, Kabaya uses these decompositions to give algebraic coordinate systems 
and furthermore goes on to study the mapping class group action \cite{Kabaya}.
These are only a few
elements of a vast literature. 

Conjecture \ref{geopw} relating the dual boundary complex of the character variety and
the sphere at infinity of the Hitchin moduli space, should be viewed as a geometric statement
reflecting the first weight-graded piece of the  $P=W$ conjecture of de Cataldo, Hausel and 
Migliorini \cite{deCataldoHauselMigliorini} \cite{Hausel}. This will be discussed a little bit more in 
Section \ref{rel-hitch} but should also be the subject of further study. 

Komyo gave the first proof of the theorem that the dual boundary complex was a sphere,
for rank $2$ systems on the projective line minus $5$ points \cite{Komyo}. He did this by
constructing an explicit compactification and writing down the dual complex. 
This provides more information than what we get in our proof of Theorem \ref{main},
because we use a large number of reduction steps iteratively replacing the character variety
by smaller open subsets.

I first heard from Mark Gross in Miami in 2012 about a statement, which he attributed
to Kontsevich, that if
$X$ is a log-Calabi-Yau variety (meaning that it has a compactification $\overline{X}=X\cup D$
such that $K_{\overline{X}}+D$ is trivial), then $\dualdel X$ should be a sphere. 
Sam Payne points out that
this idea may be traced back at least to \cite[Remark 4]{KontsevichSoibelmanHMS} 
in the situation of a degeneration. 

Gross also stated that this property
should apply to character varieties, that is to say that some or all character varieties
should be log-CY. That has apparently been known folklorically in many instances cf 
\cite{GHKK}. 

Recently, much progress
has been made. Notably, Koll\'ar and Xu have 
proven that the dual boundary of a log-CY variety is a sphere 
in dimension $4$, and they go a long way towards the proof in general
\cite{KollarXu}. 
They note that the correct statement, for general log-CY varieties, seems
to be that $\dualdel X$ should be a quotient of a sphere by a finite group. 
In our situation of character varieties, part of the statement of
Conjecture \ref{geopw} posits that this
finite quotienting doesn't happen. This is supported by our theorem, but it is hard to
say what should be expected in general. 

De Fernex, Koll\'ar and Xu have introduced a refined dual boundary complex \cite{deFernexKollarXu}, 
which is expected to be a sphere in the category of PL manifolds. That is much stronger than
just the statement about homotopy equivalence. See also Nicaise and Xu \cite{NicaiseXu}.
For character varieties, as well as for
more general cluster varieties and quiver moduli spaces, the Kontsevich-Soibelman wallcrossing
picture could be expected to be closely related to this PL sphere, more precisely the Kontsevich-Soibelman 
chambers in the base of the Hitchin fibration should to correspond to cells 
in the PL sphere. One may witness this phenomenon by explicit calculation for $SL_2$ character varieties
of the projective line minus $4$ points, under certain special choices of conjugacy classes where the character
variety is the Cayley cubic. 

Recently, Gross, Hacking, Keel and Kontsevich \cite{GHKK} building on work of Gross, Hacking and Keel
\cite{GHK}, have given an explicit combinatorial description of a boundary divisor for log-Calabi-Yau 
cluster varieties. Their description depends on a choice
of toroidal cluster coordinate patches, and the combinatorics involve 
toric geometry. It should in principle be possible to conclude from their construction that 
$\dualdel {\rm M}_B(S; C_1,\ldots , C_k)$ is a sphere, as is mentioned in 
\cite[Remark 9.12]{GHKK}. Their technique, based in essence on the Fock-Goncharov coordinate systems, 
should probably lead to a proof
in much greater generality than our Theorem \ref{main}.

\subsection{Varying the conjugacy classes}

In the present paper, we have been considering the conjugacy classes $C_1,\ldots , C_k$ as
fixed. As Deligne pointed out, it is certainly an interesting next question to ask what happens as they vary. 
Nakajima discussed it long ago \cite{Nakajima}. This has many different aspects
and it would go beyond our current scope to enter into a detailed discussion. 

I would just like to point out that the natural domain on which everything is defined
is the space of choices of $C_1,\ldots , C_k$ which satisfy the Kostov genericity Condition \ref{Kgen}.
This is an open subset of $\Gm ^k$, the complement of a divisor $K$ whose components
are defined by multiplicative monomial equalities. It therefore looks like a natural
multiplicative analogue
of the hyperplane arrangement complements which enter into the theory of higher dimensional
hypergeometric functions \cite{SchechtmanVarchenko}. The variation with parameters of the
moduli spaces ${\rm M}_B(S; C_1,\ldots , C_k)$ leads, at the very least, to some variations
of mixed Hodge structure over $\Gm ^k-K$ which undoubtedly have interesting properties.

\subsection{Acknowledgements}
I would like to thank the Fund for Mathematics at the Institute
for Advanced Study for support. 
This work was also supported in part by the 
ANR grant 933R03/13ANR002SRAR (Tofigrou). 

It is a great pleasure to thank L. Katzarkov, A. Noll and P. Pandit for all of the discussions
surrounding our recent projects, which have provided a major motivation for the present work. 
I would specially like to thank D. Halpern-Leistner, L. Migliorini and S. Payne for some
very helpful and productive discussions about this work at the Institute for Advanced Study.
They have notably been suggesting several approaches to making the reasoning more canonical,
and we hope to be able to say more about that in the future.  
I would also like to thank G. Brunerie, G. Gousin, A. Ducros, M. Gross, J. Koll\'ar, A. Komyo,
F. Loray, N. Nekrasov, M.-H. Saito, J. Weitsman, and R. Wentworth 
for interesting and informative 
discussions and suggestions. 

\subsection{Dedication}
It is a great honor to dedicate this work to Vadim Schechtman. Vadim's interests and work
have illuminated many
aspects of the intricate interplay between topology and geometry
in the de Rham theory of algebraic varieties. His work on hypergeometric functions
\cite{SchechtmanVarchenko} motivates our consideration of moduli spaces of local systems on 
higher-dimensional varieties. His work with Hinich on dga's in several papers such as
\cite{HinichSchechtman} was one of the first instances of 
homotopy methods for algebraic varieties. His 
many works on
the chiral de Rham complex have motivated wide developments in the theory of 
${\mathcal D}$-modules and local systems. The ideas generated by these threads have been suffused
throughout my own research for a long time.

\section{Dual boundary complexes}
\label{sec-dualdel}

Suppose $X$ is a smooth quasiprojective variety over $\cc$. By resolution of singularities
we may choose a normal crossings compactification $X\subset \overline{X}$ whose complementary
divisor $D:= \overline{X}-X$ has simple normal crossings. In fact, we may assume that
it satisfies a condition which might be called {\em very simple normal crossings}: if $D=\bigcup _{i=1}^m
D_i$ is the decomposition into irreducible components, then we can ask that any multiple
intersection $D_{i_1}\cap \cdots \cap D_{i_k}$ be either empty or connected. If the compactification
satisfies this condition, then we obtain a simplicial complex denoted $\dualdel X$,
the dual complex ${\mathbb D}(D)$ of the divisor $D$,
defined as follows: there are $m$ vertices $e_1,\ldots , e_m$
of $\dualdel X$, in one-to-one correspondence
with the irreducible components $D_1,\ldots , D_m$ of $D$; and a simplex spanned by $e_{i_1},\ldots , e_{i_k}$
is contained in $\dualdel X$ if and only if $D_{i_1}\cap \cdots \cap D_{i_k}$ is nonempty. 

This defines a simplicial complex, which could be considered as a simplicial set, but which for
the present purposes we shall identify with its topological realization which is the union of
the span of those simplicies in $\rr^m$ with $e_i$ being the standard basis vectors. 

The simplicial complex $\dualdel X$ goes under several different terminologies and notations. We shall
call it the {\em dual boundary complex} of $X$. It contains the purely combinatorial information about
the divisor compactifying $X$. The main theorem about it is due to Danilov \cite{Danilov}:

\begin{theorem}[Danilov]
The homotopy type of $\dualdel X$ is independent of the choice of compactification.
\end{theorem}

The papers of Stepanov \cite{Stepanov1} \cite{Stepanov2}, concerning the analogous question
for singularities, started a lot of renewed activity. 
Following these, a very instructive proof, which I first learned about from A. Ducros, 
was given by Thuillier \cite{Thuillier}. He interpreted
the homotopy type of $\dualdel X$ as being equivalent to the homotopy type of the {\em Berkovich boundary} of
$X$, namely the set of points in the Berkovich analytic space
\cite{Berkovich} 
associated to $X$ (over the trivially valued ground field), 
which are valuations centered at points outside of $X$ itself. 

Further refinements were given by Payne \cite{Payne} and de Fernex, Koll\'ar and Xu 
\cite{deFernexKollarXu}. Payne showed that
the simple homotopy type of $\dualdel X$ was invariant, and proved
several properties crucial to our arguments below. De Fernex, Koll\'ar and Xu defined in
some cases a special choice of compactification leading to a boundary complex $\dualdel X$ whose PL
homeomorphism type is invariant. Nicaise and Xu show in parallel, in the case of a degeneration
at least, that the
essential skeleton of the Berkovich space is a pseudo-manifold \cite{NicaiseXu}. 
Manon considers an embedding of ``outer space'' for character varieties, 
into the Berkovich boundary \cite{Manon}. 
These refined versions provide very interesting objects of study
but for the present paper we just use the homotopy type of $\dualdel X$. 

Our goal will be to calculate the homotopy type of the dual boundary complex of some character varieties.
To this end, we describe here a few important reduction steps allowing us to modify a variety
while keeping its dual boundary complex the same. 

One should be fairly careful when manipulating these
objects, as some seemingly insignificant changes in a variety could result in quite different boundary
complexes. For example, in the case of an isotrivial fibration it isn't enough to know the homotopy types
of the base and the fiber---essentially, the fibration should be 
locally trivial in the Zariski rather than etale 
topology in order for that kind of reasoning to work.
The space ${\bf Q}$ to be considered at the end of the paper provides an example
of this phenomenon. 

In a similar vein, I don't know of a good notion of dual boundary complex for an Artin stack. 
It is possible that a theory of Berkovich stacks could remedy this problem, but that seems 
difficult. Payne has suggested, in answer to the problem of etale isotrivial fibrations, 
to look at an equivariant notion of isotrivial fibration 
which could give a natural group action on a dual boundary complex such that the
quotient would be meaningful. This type of theory might give an alternate approach to some of
our problems.

Let us get now to the basic properties of dual boundary complexes. 
The first step is to note that if $U\subset X$ is an open subset of a smooth quasiprojective variety, then
we obtain a map $\dualdel X \rightarrow \dualdel U$. 

\begin{lemma}[Payne]
If $X$ is an irreducible smooth projective variety and $Z\rightarrow X$ is obtained by blowing up a smooth
center, then it induces a homotopy equivalence on dual boundary complexes $\dualdel Z \sim \dualdel X$.
\end{lemma}
\begin{proof}
See \cite{Payne}, where more generally boundary complexes of singular varieties
are considered but we only need the smooth case.
\end{proof}

It follows from this lemma that
if $U\subset X$ is an open subset of a smooth quasiprojective variety, it induces a natural
map of boundary complexes $\partial X\rightarrow \partial U$.
Indeed, for that we may assume by resolving singularities and applying the previous lemma,
that $U$ is the complement of a divisor $B\subset X$, and furthermore there is a very simple
normal crossings compactification $\overline{X}=X\cup D$ such that $B\cup D$ also has very simple
normal crossings. Then $\dualdel U$ is the dual complex of the divisor $B\cup D$, which
contains  $\dualdel X$, the dual complex of $D$, as a subcomplex. 

Following up on this idea, here is our main reduction lemma:

\begin{lemma}
\label{negligeable}
Suppose $U\subset X$ is an open subset of an irreducible smooth quasiprojective variety, obtained by removing
a smooth irreducible closed subvariety of smaller dimension $Y=X-U\subset X$. Suppose that $\dualdel Y \sim \ast$ is contractible.
Then the map $\dualdel X\rightarrow \dualdel U$ is a homotopy equivalence. 
\end{lemma}
\begin{proof}
Let $X^{{\rm Bl}Y}$ be obtained by blowing up $Y$. From the previous lemma, 
$\dualdel X^{{\rm Bl}Y} \sim \dualdel X$.
Let ${\rm Bl}(Y)\subset X^{{\rm Bl}Y}$ be the inverse image of $Y$. 
It is an irreducible smooth divisor, and $U$ is also the complement of
this divisor in $X^{{\rm Bl}Y}$. By resolution
of singularities we may choose a compactification 
$\overline{X^{{\rm Bl}Y}}$ such that the boundary divisor $D$, plus the
closure $B:= \overline{{\rm Bl}(Y)}$,
form a very simple normal crossings divisor. This combined divisor is therefore a boundary
divisor for $U$, so
$$
\dualdel U \sim {\mathbb D}( D \cup B ).
$$
Now this bigger dual complex ${\mathbb D}( D \cup B)$ has one
more vertex than ${\mathbb D}(D)$, corresponding to the irreducible component
$B$. The star of this vertex is the cone over $\dualdel {\rm Bl}(Y) = {\mathbb D}(B\cap D)$.
The cone is attached to ${\mathbb D}(D)$ via  its base ${\mathbb D}(B\cap D)$,  to give
${\mathbb D}(B\cup D)$. 

We would like to show that $\dualdel {\rm Bl}(Y)\sim \ast$. The first step is to notice that 
${\rm Bl}(Y)\rightarrow Y$ is the projective space bundle associated to the vector bundle 
$N_{Y/X}$ over $Y$. 

We claim in general that if $V$ is a vector bundle over a smooth
quasiprojective variety $Y$, then $\dualdel ({\mathbb P}(V))\sim \dualdel (Y)$. 
The proof of this claim is that there exists a normal crossings compactification 
$\overline{Y}$ of $Y$ such that the vector bundle $V$ extends to a vector bundle
on $\overline{Y}$. That may be seen by choosing a surjection from the dual
of a direct sum of very ample line bundles to $V$, getting $V$ as the pullback of
a tautological bundle under a map from $Y$ to a Grassmanian. The compactification 
may be chosen so that the map to the Grassmanian extends. 
We obtain a compactification of ${\mathbb P}(V)$ wherein the boundary divisor is
a projective space bundle over the boundary divisor of $Y$, and with these choices
$\dualdel {\mathbb P}(V)=\dualdel Y$. It follows from Danilov's theorem that
for any other choice, there is a homotopy equivalence.

Back to our situation where ${\rm Bl}(Y)={\mathbb P}(N_{Y/X})$, and assuming that
$\dualdel Y\sim \ast$, we conclude that $\dualdel {\rm Bl}(Y)\sim \ast$ too. 
Therfore the dual complex ${\mathbb D}(B\cap D)$ is contractible.

Now $\dualdel U = {\mathbb D} (B\cup D)$ is obtained by attaching to ${\mathbb D}(D)$
the cone over ${\mathbb D}(B\cap D)$. As we have seen above ${\mathbb D}(B\cap D)$
is contractible, so coning it off doesn't change the homotopy type. This shows that
the map
$$
\dualdel X = {\mathbb D}(D) \rightarrow {\mathbb D} (B\cup D)= \dualdel U
$$
is a homotopy equivalence. 
\end{proof}

In order to use this reduction, we need a criterion for the condition $\dualdel Y \sim \ast$.
Note first the following general property of compatibility with products.

\begin{lemma}[Payne]
\label{join}
Suppose $X$ and $Y$ are smooth quasiprojective varieties. Then $\dualdel (X\times Y)$ is the
{\em join} of $\dualdel (X)$ and $\dualdel (Y)$, in other words we have a homotopy cocartesian
diagram of spaces
$$
\begin{array}{ccc}
\dualdel (X)\times \dualdel (Y) & \rightarrow & \dualdel (Y) \\
\downarrow && \downarrow \\
\dualdel (X) & \rightarrow & \dualdel (X\times Y)\,  .
\end{array} 
$$
\end{lemma}
\begin{proof}
This is \cite[Lemma 6.2]{Payne}. 
\end{proof}

\begin{corollary}
\label{affine}
Suppose $Y$ is a smooth quasiprojective variety. Then $\dualdel (\aaa ^1\times Y) \sim \ast$. 
\end{corollary}
\begin{proof}
Setting $X:= \aaa^1$ in the previous lemma, we have $\dualdel (X)\sim \ast$, so in the
homotopy cocartesian diagram the top arrow is an equivalence and the left vertical arrow
is the projection to $\ast$; therefore the homotopy pushout is also $\ast$. 
\end{proof}

\begin{proposition}
\label{decomp}
Suppose $U\subset X$ is a nonempty open subset of a smooth irreducible 
quasiprojective variety, and suppose
the complement $Z:= X-U$ has a decomposition into locally closed subsets $Z_j$ such that
$Z_j\cong \aaa^1\times Y_j$. Suppose that this decomposition can be ordered into
a stratification, that is to say there is a total order on the indices such that
$\bigcup _{j\leq a}Z_j$ is closed for any $a$. Then $\dualdel (X)\sim \dualdel (U)$. 
\end{proposition}
\begin{proof}
We first prove the proposition under the additional hypothesis that the $Y_j$ are smooth. 
Proceed by induction on the number of pieces in the decomposition. Let $Z_0$ be the
lowest piece in the ordering. The ordering hypothesis says that $Z_0$ is closed in $X$.
Let $X':= X-Z_0$. Now $U\subset X'$ is the complement
of a subset $Z' = \bigcup _{j>0}Z_j$ decomposing in the same way, with a smaller number of
pieces, so by induction we know that 
$\dualdel (X')\sim \dualdel (U)$. 

By hypothesis $Z_0\cong \aaa ^1\times Y_0$.
Lemma \ref{affine} tells us that $\dualdel (Z_0)\sim \ast$ and now Lemma \ref{negligeable}
tells us that $\dualdel (X)\sim \dualdel (X')$, so $\dualdel (X)\sim \dualdel (U)$. 
This completes the proof of the proposition under the hypothesis that $Y_j$ are smooth. 

Now we prove the proposition in general. Proceed as in the first paragraph of the proof with
the same notations: by induction we may assume that $\dualdel (X')\sim \dualdel (U)$
where $X'=X-Z_0$ such that $Z_0$ is closed and isomorphic to $\aaa^1\times Y_0$. 
Choose a totally ordered stratification of $Y_0$ by smooth locally closed subvarieties
$Y_{0,i}$. Set $Z_{0,i}:= \aaa^1\times Y_{0,i}$. This collection of subvarieties
of $X$ now satisfies the hypotheses of the proposition and the pieces are smooth.
Their union is $Z_0$ and its complement in $X$ is the open subset $X'$. Thus,
the first case of the proposition treated above tells us that $\dualdel (X)\sim \dualdel (X')$. 
It follows that $\dualdel (X) \sim \dualdel (U)$, completing the proof. 
\end{proof}

\noindent
{\em Caution:} A simple example shows that the condition of ordering, in the statement of the
propostion, is necessary. Suppose  $X$ is a smooth
projective surface containing two projective lines $D_1,D_2\subset X$ such that their intersection 
$D_1\cap D_2 = \{ p_1,p_2\}$ consists of two distinct points. Then we could look at 
$Z_1=D_1-\{ p_1\}$ and $Z_2=D_2-\{ p_2\}$. Both $Z_1$ and $Z_2$ are affine lines.
Setting $U:= X-(D_1\cup D_2)=X-(Z_1\cup Z_2)$ we get an open set which is the complement
of a subset $Z=Z_1\sqcup Z_2$ decomposing into two affine lines; but $\dualdel X=\emptyset$ whereas 
$\dualdel U \sim S^1$. 

\section{Hybrid moduli stacks of local systems}
\label{sec-hybrid}

The moduli space of local systems is different from the moduli stack, even at the points
corresponding to irreducible local systems. Indeed, the open substack of the moduli stack
parametrizing irreducible $GL_r$-local systems is a $\Gm$-gerbe over the corresponding open subset
of the moduli space. Even by considering $SL_r$-local systems we can only reduce this to being a
$\mu _r$-gerbe. 

However, it is usual and convenient to consider the moduli space instead. 
In this section, we mention a construction allowing to define what we 
call\footnote{This is not new but I don't remember where it is from, so
no claim is made to originality.}
a {\em hybrid moduli stack}
in which the central action is divided out, making it so that for irreducible points it is the
same as the moduli space. 

Our initial discussion will use some simple $2$-stacks, however the
reader wishing to avoid these is may refer to Proposition \ref{alternate} 
which gives an equivalent definition in
more concrete terms. 

Consider a reductive group $G$ with center $Z$. 
The fibration sequence of $1$-stacks
$$
BZ \rightarrow BG \rightarrow B(G/Z)
$$
may be transformed into the cartesian diagram
\begin{equation}
\label{kzdiag}
\begin{array}{ccc}
BG &\rightarrow & B(G/Z) \\
\downarrow & & \downarrow\\
\ast & \rightarrow &  K(Z,2)
\end{array}
\end{equation}
of Artin $2$-stacks on the site ${\rm Aff}^{ft, et}_{\cc}$ of affine schemes 
of finite type over $\cc$ with the etale topology. 
 
Suppose now that $S$ is a space or higher stack. Then we may consider the relative
mapping stack
$$
M(S, G):= \underline{Hom}(S, B(G/Z)/K(Z,2)) \rightarrow K(Z,2).
$$
It may be defined as the fiber product forming the middle arrow in the following
diagram where both squares are cartesian:  
$$
\begin{array}{ccccc}
\underline{Hom}(S,BG)& \rightarrow & M(S,G) &\rightarrow & \underline{Hom}(S,B(G/Z)) \\
\downarrow & & \downarrow & & \downarrow\\
\ast & \rightarrow & K(Z,2) & \rightarrow &  \underline{Hom}(S,K(Z,2))
\end{array} .
$$
Here the bottom right map is the ``constant along $S$'' construction induced by pullback 
along $S\rightarrow \ast$. 

The bottom left arrow $\ast \rightarrow K(Z,2)$ is the universal $Z$-gerbe, so its pullback
on the upper right is again a $Z$-gerbe. We have thus constructed a stack
$M(S,G)$ over which $\underline{Hom}(S,BG)$ is a $Z$-gerbe. From the definition it is
{\em a priori} a $2$-stack, and indeed $M(\emptyset , G)=K(Z,2)$, but the following
alternate characterization tells us that $M(S,G)$ is usually a usual $1$-stack. 

\begin{proposition}
\label{alternate}
Suppose $S$ is a nonempty connected CW-complex with basepoint $x$. Then 
the hybrid moduli stack may be expressed as the stack-theoretical quotient
$$
M(S,G) = {\rm Rep}(\pi _1(S,x), G) \stackmod (G/Z) .
$$
In particular, it is an Artin $1$-stack. 
\end{proposition}
\begin{proof}
The representation space may be viewed as a mapping stack
$$
{\rm Rep}(\pi _1(S,x), G) = \underline{Hom}((S,x), (BG,o)).
$$
Consider the big diagram
$$
\begin{array}{ccccc}
\underline{Hom}((S,x), (BG,o)) & \rightarrow & \underline{Hom}((S,x), (B(G/Z),o)) & & \\
\downarrow & & \downarrow & & \\
\ast & \rightarrow & \underline{Hom}((S,x), (K(Z,2),o)) & \rightarrow & \ast \\
\downarrow & & \downarrow & & \downarrow \\
K(Z,2) & \rightarrow & \underline{Hom}(S, K(Z,2)) & \rightarrow & K(Z,2)
\end{array}
$$
where the bottom right map is evaluation at $x$. The pointed mapping $2$-stack in the
middle is defined by the condition that the bottom right square is homotopy cartesian. 
The composition along the bottom is the identity, so if we take the homotopy fiber
product on the bottom left, the full bottom rectangle is a pullback too so that homotopy
fiber product would be $\ast$ as is written in the diagram. In other words, the bottom left
square is also homotopy cartesian. The middle horizontal map on the left sends the
point to the map $S\rightarrow o \hookrightarrow K(Z,2)$, indeed it is constant along $S$ because
it comes from pullback of the bottom left map, and its value at $x$ is $o$ because of the
right vertical map. Now, the upper left square is homotopy cartesian, just the result
of applying the pointed mapping stack to the diagram \eqref{kzdiag}. It follows that the
whole left rectangle is homotopy cartesian. 

Consider, on the other hand, the diagram
$$
\begin{array}{ccccc}
\underline{Hom}((S,x), (BG,o)) & \rightarrow & \underline{Hom}((S,x), (B(G/Z),o)) &
\rightarrow & \ast \\
\downarrow & & \downarrow & & \downarrow \\
M(S,G) & \rightarrow & \underline{Hom}(S, B(G/Z)) & \rightarrow & B(G/Z) \\
\downarrow & & \downarrow & &  \\
K(Z,2) & \rightarrow & \underline{Hom}(S, K(Z,2)) & & .
\end{array}
$$
The bottom square is homotopy-cartesian by the definition of $M(S,G)$.
We proved in the previous paragraph that
the full left rectangle is homotopy cartesian. In this $2$-stack situation note that
a commutative rectangle constitutes a piece of data rather than just a property. In this case,
these data for the left squares are obtained by just considering the equivalence found in
the previous paragraph, from $\underline{Hom}((S,x), (BG,o))$ to the homotopy pullback
in the full left rectangle which is the same as the composition of the homotopy pullbacks
in the two left squares. In particular, the upper left square is homotopy-cartesian. 
It now follows that the upper full rectangle is homotopy-cartesian. That exactly says
that we have an action of $G/Z$ on $\underline{Hom}((S,x), (BG,o))= {\rm Rep}(\pi _1(S,x), G)$
and $M(S,G)$ is the quotient. 
\end{proof}

The hybrid moduli stacks also satisfy the same glueing or factorization property as the
usual ones.

\begin{lemma}
\label{glueing}
Suppose $S=S_1\cup S_2$ with $S_{12}:= S_1\cap S_2$ excisive. Then 
$$
M(S,G) \cong M(S_1,G)\times _{M(S_{12},G)} M(S_2,G).
$$
\end{lemma}
\begin{proof}
The mapping stacks entering into the definition of $M(S,G)$ as a homotopy
pullback, satisfy this
glueing property. Notice that this is true even for the constant functor which 
associates to any $S$ the stack $K(Z,2)$. The homotopy pullback therefore also satisfies
the glueing property since fiber products commute with other fiber products. 
\end{proof}

Suppose $G=GL_r$
so $Z=\Gm$ and $G/Z = PGL_r$, and suppose $S$ is a connected CW-complex.
Let 
$$
\underline{Hom}(S,BGL_r)^{\rm irr} \subset \underline{Hom}(S,BGL_r)
$$
denote the open substack of irreducible local systems. It is classical that 
the stack $\underline{Hom}(S,BGL_r)$ has a {\em coarse moduli space} ${\rm M}_B (S,GL_r)$,
and that the open substack
$\underline{Hom}(S,BGL_r)^{\rm irr}$ is a $\Gm$-gerbe over the corresponding open subset
of the coarse moduli space ${\rm M}_B(S, G)^{\rm irr}$. 

\begin{proposition}
In the situation of the previous paragraph,
we have a map 
$$
M(S,GL_r)\rightarrow {\rm M}_B (S,GL_r)
$$
which restricts to an isomorphism 
$$
M(S,GL_r)^{\rm irr}\cong {\rm M}_B (S,GL_r)^{\rm irr}
$$
between
the open subsets parametrizing
irreducible local systems. 
\end{proposition}

The same holds for $G=SL_r$. 

\begin{lemma}
The determinant map $GL_r \stackrel{{\rm det}}{\rightarrow} \Gm$ induces a cartesian diagram
$$
\begin{array}{ccc}
M(S, SL_r) &\rightarrow & M(S,GL_r) \\
\downarrow & & \downarrow\\
\ast & \rightarrow &  M(S, \Gm )
\end{array}
$$
which essentially says that $M(S,SL_r)$ is the substack of $M(S,GL_r)$ parametrizing local systems
of trivial determinant. Note that $M(S, \Gm )$ is isomorphic to the quasiprojective variety 
${\rm Hom}(H_1(S), \Gm )$. 
\end{lemma}

In what follows, we shall use these stacks $M(S,GL_r)$ which we 
call
{\em hybrid moduli stacks} as good replacements intermediary between
the moduli stacks of local systems and their coarse moduli spaces.

\section{Boundary conditions}
\label{sec-local}

Let $S$ denote a $2$-sphere with $k$ open disks removed. It has $k$ boundary
circles denoted $\xi _1,\ldots , \xi _k\subset S$ and 
$$
\partial S = \xi _1\sqcup \cdots \sqcup \xi _k .
$$
From now on we consider rank $2$ local systems on this surface $S$. 

Fix complex numbers $c_1,\ldots , c_k$ all different from $0$, $1$ or $-1$. Let 
$$
C_i := \left\{ P \left(
\begin{array}{cc}
c_i & 0 \\
0 & c_i ^{-1} 
\end{array}
\right) P^{-1} \right\}
$$
denote the conjugacy class of matrices with eigenvalues $c_i, c_i^{-1}$. 

Consider the hybrid moduli stack $M(S, GL_2)$ constructed above,
and let  
$$
M(S;C_{\cdot}) \subset M(S,GL_2)
$$
denote the closed substack consisting of local systems such that the monodromy transformation
around $\xi _i$ is in the conjugacy class $C_i$.  See \cite{Letellier}. 

If we choose a basepoint $x\in S$ and paths $\gamma _i$ going from $x$ out by straight paths to the boundary
circles, around once and then back to $x$, then $\pi _1(S,x)$ is generated by the $\gamma _i$
subject to the relation that their product is the identity. 

Therefore, the moduli stack of framed
local systems is the affine variety 
$$
\underline{Hom}((S,x), (BGL_2, o)) = 
{\rm Rep}(\pi _1(S,x), GL_2) 
$$
$$
= \{ (A_1,\ldots , A_k)\in (GL_2)^k \mbox{ s.t. } A_1\cdots A_k =1\} .
$$
The unframed moduli stack is the stack-theoretical quotient
$$
\underline{Hom}(S, BGL_2) = {\rm Rep}(\pi _1(S,x), GL_2)  \stackmod  GL_2 
$$
by the action of simultaneous conjugation. 

The center $\Gm \subset GL_2$ acts trivially on ${\rm Rep}(\pi _1(S,x), GL_2)$ so
the action of $GL_2$ there 
factors through an action of $PGL_2$.

Proposition \ref{alternate} may be restated as

\begin{lemma}
The hybrid moduli stack $M(S,GL_2)$ may be described as
the stack-theoretical quotient
$$
M(S,GL_2) = 
{\rm Rep}(\pi _1(S,x), GL_2)  \stackmod PGL_2 .
$$
\end{lemma}

Let ${\rm Rep}(\pi _1(S,x), GL_2; C_{\cdot})\subset {\rm Rep}(\pi _1(S,x), GL_2)$
denote the closed subscheme of representations which send $\gamma _i$ to the conjugacy 
class $C_i$. These conditions are equivalent to the equations ${\rm Tr}(\rho (\gamma _i))=c_i+c_i^{-1}$.
We have
$$
{\rm Rep}(\pi _1(S,x), GL_2; C_{\cdot}) = 
\{ (A_1,\ldots , A_k)\,  \mbox{s.t.}\,  A_i\in C_i \mbox{ and } A_1\cdots A_k =1\} .
$$

\begin{corollary}
\label{hybquot}
The hybrid moduli stack with fixed conjugacy classes is given by 
$$
M(S; C_{\cdot} ) = {\rm Rep}(\pi _1(S,x), GL_2; C_{\cdot}) \stackmod PGL_2 
$$
$$
=\{ (A_1,\ldots , A_k) \mbox{ s.t. } A_i\in C_i \mbox{ and } A_1\cdots A_k =1\} \stackmod PGL_2.
$$
It is also isomorphic to the stack one would have gotten by using the group $SL_2$ rather than
$GL_2$. 
\end{corollary}
\begin{proof}
Our conjugacy classes have been defined as having determinant one. Since the $\gamma _i$ generate
the fundamental group, if the $\rho (\gamma _i)$ have determinant one then the representation 
$\rho$ goes into $SL_2$. As $PGL_2=PSL_2$, the hybrid moduli stack for $GL_2$ is the same as
for $SL_2$. 
\end{proof}

Recall the following Kostov-genericity condition \cite{Kostov} on the choice of the numbers $c_i$.  

\begin{condition}
\label{Kgen}
For any choice of $\epsilon _1,\ldots , \epsilon _k\in \{ 1,-1\}$ the product
$$
c_1^{\epsilon _1} \cdots c_k^{\epsilon _k}
$$
is not equal to $1$. 
\end{condition}

The following basic lemma has been observed by Kostov and others. 

\begin{lemma}
\label{Kstab}
If Condition \ref{Kgen} is satisfied then any 
representation in 
${\rm Rep}(\pi _1(S,x), GL_2; C_{\cdot})$ is irreducible. 
In particular, the automorphism group of the corresponding $GL_2$ local system is the central
$\Gm$.   
\end{lemma}

The set of $(c_1,\ldots , c_k)$ satisfying this condition is a
nonempty open subset of $(\Gm - \{ 1,-1\} )^k$. We also speak of the same condition for the
sequence of conjugacy classes $C_{\cdot}$.

\begin{proposition}
\label{variety}
Suppose $C_{\cdot}$ satisfy Condition \ref{Kgen}. The hybrid moduli stack 
$M(S; C_{\cdot} )$ is an irreducible smooth affine variety. It is
equal to the coarse, which is indeed fine, moduli space
${\rm M}_B(S; C_1,\ldots , C_k)$
of local systems with our given conjugacy classes. 
\end{proposition}
\begin{proof}
The representation space ${\rm Rep}(\pi _1(S,x), GL_2; C_{\cdot})$ is an affine variety,
call it ${\rm Spec}(A)$,
on which the group $PGL_2$ acts. 
The moduli space is by definition
$$
{\rm M}_B(S; C_1,\ldots , C_k):= {\rm Spec}(A^{PGL_2}).
$$ 
By Lemma \ref{Kstab} and using the hypothesis \ref{Kgen}
it follows that the stabilizers of the action are trivial.
Luna's etale slice theorem
(see \cite{Drezet}) implies that the quotient map 
$$
{\rm Spec}(A)\rightarrow {\rm Spec}(A^{PGL_2})
$$
is an etale fiber bundle with fiber $PGL_2$. Therefore this quotient is also the 
stack-theoretical quotient:
$$
{\rm Spec}(A^{PGL_2})= {\rm Spec}(A)\stackmod PGL_2 .
$$
By Corollary \ref{hybquot} that stack-theoretical quotient is $M(S;C_{\cdot})$, completing
the identification between the hybrid moduli stack and the moduli space
required for the proposition. 

Smoothness of the moduli space has
been noted in many places, see for example \cite{HauselLetellierRodriguez} \cite{Letellier}.
Irreducibility is proven in a general context in \cite{HauselLetellierRodriguez2} \cite{Letellier}
as a consequence of computations of $E$-polynomials, and a different proof is given in 
\cite{ASoibelman} using moduli stacks of parabolic bundles. In our case
irreducibility could also be obtained by including dimension estimates for the subvarieties which will be
removed in the course of our overall discussion.  
\end{proof}

This proposition says that our hybrid moduli stack $M(S; C_{\cdot} )$ is the same as the
usual moduli space. A word of caution is necessary: we shall also be using 
$M(S',C_{\cdot})$ for subsets $S'\subset S$, and those are in general 
stacks rather than schemes, for example when Condition \ref{Kgen} doesn't hold over $S'$.

\section{Interior conditions and factorization}
\label{sec-interior}

We  now define some conditions concerning what happens in the interior of the surface $S$. 
These conditions will serve to define a stratification of $M(S; C_{\cdot} )$. 
The biggest open stratum denoted $M'$, treated in detail in Section \ref{sec-main}, 
turns out to be the main piece, contributing the essential structure of the 
dual boundary complex. The smaller strata will be negligeable for the dual boundary complex,
in view of Lemmas \ref{negligeable} and \ref{affine} as combined in Proposition \ref{decomp}. 

Divide $S$ into closed regions denoted $S_2,\ldots , S_{k-1}$
such that $S_{i}\cap S_{i+1} = \rho _i$ is a circle for $2\leq i \leq k-2$, and 
the regions are otherwise disjoint. 
We assume that $S_i$ encloses the boundary circle $\xi _i$, so it is a $3$-holed
sphere with boundary circles $\rho _{i-1}$, $\xi _i$ and $\rho _i$. 
The orientation of $\rho _{i-1}$ is reversed when it is viewed from inside 
$S_i$. The end piece $S_2$ has boundary circles $\xi _1$, $\xi _2$ and $\rho _2$ while
the end piece $S_{k-1}$ has boundary circles $\rho _{k-2}$, $\xi _{k-1}$ and $\xi _i$. 
This is a ``pair of pants'' decomposition.

Factorization properties, related to
chiral algebra cf \cite{FrancisGaitsgory} \cite{FrenkelBenZvi}, are a kind of descent. 
We will be applying the factorization property of Theorem \ref{factorization} to the
decomposition of our surface into pieces $S_i$.  This classical technique in geometric topology
was also used extensively in the study of the Verlinde formula. The factorization
is often viewed as coming from
a degeneration of the curve into a union of rational lines with three marked points. 

For our argument it will be important to consider strata of the moduli space defined by
fixing additional combinatorial data with respect to our decomposition. 
To this end, let us consider some nonempty subsets $\sigma _i \subset \{ 0,1\}$
for $i=2,\ldots , k-1$, and conjugacy-invariant subsets $G_2,\ldots , G_{k-2}\subset SL_2$.
We denote by $\alpha = (\sigma _1,\ldots , \sigma _{k-1}; G_2,\ldots , G_{k-2})$
this collection of data. The subsets $G_i$ will impose conditions on 
the monodromy around the circles $\rho _i$, while the $\sigma _i$ will correspond to
the following {\em stability condition} on the restrictions of our local system to
$S_i$. Recall that a local system $V\in M(S;C_{\cdot})$ is required to have monodromy
around $\xi _i$ with eigenvalues $c_i$ and $c_i^{-1}$. We are making a choice of orientation
of these boundary circles, and $c_i\neq c_i^{-1}$ by hypothesis, 
so the $c_i^{-1}$ eigenspace is a well-defined rank $1$ subspace of $V|_{\xi _i}$. 

\begin{definition}
\label{stab-def}
We say that a local system $V|_{S_i}$ on $S_i$, satisfying the conjugacy class
condition, is {\em unstable} if there exists a rank $1$ subsystem
$L\subset V|_{S_i}$ such that the monodromy of $L$ around $\xi _i$ is $c_i^{-1}$. Say that 
$V|_{S_i}$ is {\em stable} 
otherwise. 
\end{definition}

An irreducible local system $V|_{S_i}$ 
is automatically stable; one which decomposes as a direct sum is automatically unstable.
If $V|_{S_i}$ is a nontrivial extension with a unique rank $1$ subsystem $L$, then
$V|_{S_i}$ is unstable if $L|_{\xi _i}$ is the $c_i^{-1}$-eigenspace of the monodromy, 
whereas it is stable if $L|_{\xi _i}$ is the $c_i$-eigenspace. We will later express these
conditions more concretely in terms of vanishing or nonvanishing of a certain matrix
coefficient. 

\begin{definition}
Let $M^{\alpha}(S; C_{\cdot})\subset M(S; C_{\cdot})$ denote the locally closed substack
of local systems $V$ satisfying the following conditions: 
\begin{itemize}

\item 
if $\sigma _i = \{ 0\}$ then $V|_{S_i}$ is required to be unstable; if $\sigma _i = \{ 1\}$ then
it is required to be stable; and if $\sigma _i = \{ 0,1\}$ then there is no condition; and

\item 
the monodromy of $V$ around $\rho _i$ should lie in $G_i$.

\end{itemize}

Consider a subset $S'\subset S$ made up of some or all of the $S_i$ or the circles. 
Let $M ^{\alpha} (S';C_{\cdot})$ denote the moduli stack of local systems on $S'$ satisfying the above
conditions where they  make sense (that is, for the restrictions to those subsets which
are in $S'$). 
\end{definition}

In the case of the inner boundary circles we may just use the notation
$M^{\alpha}  (\rho _i)$ since the choices of conjugacy classes $C_i$, corresponding
circles $\xi _i$, don't intervene. 

In the case of $S_i$, only the conjugacy class $C_i$ matters so we may use the notation 
$M ^{\alpha} (S_i;C_i)$.

Suppose $S'\subset S$ is connected and $x\in S'$. Let 
$$
{\rm Rep}^{\alpha}(\pi _1(S',x), GL_2; C_{\cdot}) \subset 
{\rm Rep}(\pi _1(S',x), GL_2)
$$
denote the locally closed subscheme of representations which
satisfy conjugacy class conditions corresponding to $C_{\cdot}$
and the conditions corresponding to $\alpha$, that is  
to say whose corresponding local systems are in $M ^{\alpha} (S';C_{\cdot})$. 
Proposition \ref{alternate} says:

\begin{lemma}
The simultaneous conjugation action of $GL_2$ on the space of representations
${\rm Rep}^{\alpha}(\pi _1(S',x), GL_2; C_{\cdot})$ factors through an action of $PGL_2$ and
$$
M ^{\alpha} (S';C_{\cdot})={\rm Rep}^{\alpha}(\pi _1(S',x), GL_2; C_{\cdot})\stackmod PGL_2
$$
is the stack-theoretical quotient. 
\end{lemma}

The hybrid moduli stacks allow us to state a glueing or {\em factorization property},
expressing the fact that a local system $L$ on $S$ may be viewed as being obtained
by glueing together its pieces $L|_{S_i}$ along the circles $\rho _i$. 

\begin{theorem}
\label{factorization}
We have the following expression using homotopy fiber products of stacks:
$$
M ^{\alpha} (S;C_{\cdot})= M^{\alpha}  (S_2;C_{\cdot})\times _{M^{\alpha}  (\rho _i)} M^{\alpha}  (S_3
;C_{\cdot}) 
\cdots \times _{M^{\alpha}  (\rho _{k-2})}
M^{\alpha}  (S_{k-1};C_{\cdot}).
$$
\end{theorem}
\begin{proof}
Apply Lemma \ref{glueing}. 
\end{proof}

\begin{corollary}
Suppose the requirements given for the boundary pieces of $\partial S'$ (which are
circles either of the form $\xi _i$ or $\rho _i$) satisfy Condition \ref{Kgen}
for $S'$. Then
the moduli stack $M^{\alpha}  (S';C_{\cdot})$ is in fact a quasiprojective variety.
\end{corollary}
\begin{proof}
This follows from Proposition \ref{variety} applied to $S'$. 
\end{proof}

\section{Universal objects}
\label{sec-universal}

Let us return for the moment to the general situation of Section \ref{sec-hybrid},
of a space $S$ and a group $G$.
If $x\in S$ is a basepoint, then we obtain a principal $(G/Z)$-bundle over
$\underline{Hom}(S,B(G/Z))$, and this pulls back to a principal $(G/Z)$-bundle
denoted $F(S,x)\rightarrow M(S,G)$. It may be viewed as the bundle of frames
for the local systems, up to action of the center $Z$. 

If $y\in S$ is another point, and $\gamma$ is a path from $x$ to $y$ then it 
gives an isomorphism of principal bundles $F(S,x)\cong F(S,y)$ over $M(S,G)$. 
In particular, $\pi _1(S,x)$ acts on $F(S,x)$ in a tautological representation. 

Suppose $S=S_1\cup S_2$ such that the intersection $S_{12}=S_1\cap S_2$ is connected. 
Choose a basepoint $x\in S_{12}$. This yields principal $(G/V)$-bundles
$F(S_1,x)$ and $F(S_2,x)$ over $M(S_1,G)$ and $M(S_2,G)$ respectively. 
The fundamental group $\pi _1(S_{12},x)$ acts on both of these. 
We may restate the glueing property of Lemma \ref{glueing}
in the following way. 

\begin{proposition}
\label{principal}
We have an isomorphism of stacks lying over the product  
$M(S_1,G)\times M(S_2,G)$,
$$
M(S,G) \cong {\rm Iso}_{\pi _1(S_{12},x)-G/V}(p_1^{\ast} F(S_1,x),
p_2^{\ast}F(S_2,x))
$$
where on the right is the stack of isomorphisms, relative to $M(S_1,G)\times M(S_2,G)$, 
of principal $G/V$-bundles provided with 
actions of $\pi _1(S_{12},x)$.
\end{proposition}

Return now to the notation from the immediately preceding sections. There are several ways of
dividing our surface $S$ into two or more pieces, various of which shall be used in the 
next section. 

Choose basepoints $x_i$ in the interior of $S_i$, and $s_i$ on the boundary circles
$\rho _i$. Connect them by paths, nicely arranged with respect to the other paths
$\gamma _i$. Then, over any subset $S'$ containing a basepoint $x_i$, we obtain
a principal $PGL_2$-bundle $F(S',x_i)\rightarrow M(S',C_{\cdot})$, and the same for $s_i$. Our paths,
when in $S'$, give isomorphisms between these principal bundles. 

It will be helpful to think of the description of glueing given by Proposition \ref{principal},
using these basepoints and paths. The following local triviality property is useful. 

\begin{lemma}
\label{lem62}
Suppose $S'$ has at most one boundary circle of the form $\rho_i$,
and suppose that the conjugacy classes determining the moduli problem on $M^{\alpha}(S',C_{\cdot})$
satisfy Condition \ref{Kgen},
and suppose that $x\in S'$ is one of our basepoints. 
Then the principal $PGL_2$-bundle $F(S',x)\rightarrow 
M^{\alpha}(S',C_{\cdot})$
is locally trivial in the Zariski topology of the moduli space  $M^{\alpha}(S',C_{\cdot})$, and Zariski
locally
$F(S',x)$ 
may be viewed as the projective frame bundle of a rank $2$ vector bundle.
\end{lemma}
\begin{proof}
Consider a choice of three loops $(\gamma _{j_1},\gamma _{j_2},\gamma _{j_3})$ 
and a choice of one of the two eigenvalues of the 
conjugacy class $C_{j_1}$, $C_{j_1}$, or $C_{j_1}$
for each of them. This gives three rank $1$ eigenspaces in $V_x$
for any local system $V$. Over the Zariski open subset of the moduli space where
these three subspaces are distinct, they provide the required projective frame. 
Notice that the eigenspaces of the $\gamma _j$ cannot all be aligned since 
these loops generate the fundamental group of $S'$, by the hypothesis that there is at most
one other boundary circle $\rho _i$. Therefore, as the our choices
of triple of loops and triple of eigenvalues range over the possible ones, these Zariski open subsets
cover the moduli space. We get the required frames. A framed $PGL_2$-bundle comes from a vector bundle
so $F(S',x)$ locally comes from a $GL_2$-bundle. 
\end{proof}

\section{Splitting along the circle $\rho _i$}
\label{sec-splitting}

In this section we consider one of the circles $\rho_i$ which divides $S$ into two pieces. 
Let 
$$
S_{<i}:=\bigcup _{j<i}S_j\; , \;\;\;\;
S_{>i}:=\bigcup _{j>i}S_j , 
$$
and similarly define $S_{\leq i}$ and $S_{\geq i}$. 
We have the decomposition 
$$
S=S_{\leq i} \cup S_{>i}
$$
into two pieces intersecting along the circle $\rho _i$. 
Thus,
$$
M^{\alpha}(S;C_{\cdot}) = M^{\alpha}(S_{\leq i};C_{\cdot})
\times _{M^{\alpha}(\rho _i)} M^{\alpha}(S_{>i};C_{\cdot}).
$$
This factorization will allow us to analyze strata where $G_i$ is
a unipotent or trivial conjugacy class. The following condition will be in effect: 

\begin{condition}
\label{verygen}
We assume that the sequence of conjugacy classes $C_1,\ldots , C_k$ is {\em very
generic}, meaning that  for any $i$ the partial sequences $C_1,\ldots , C_i$ and 
$C_i,\ldots , C_k$ satisfy Condition \ref{Kgen}, and they also satisfy that condition
if we add the scalar matrix $-1$. 
That is to say, no product of eigenvalues or their inverses should be
equal to either $1$ or $-1$.
\end{condition}

Suppose that $G_i=\{ 1\}$. Then $M^{\alpha} (\rho _i)=B(PGL_2)$. 
On the other hand, Condition \ref{verygen} means that
the sequences of conjugacy classes defining the
moduli problems on $S_{\leq i}$ and $S_{>i}$ themselves satisfy Condition \ref{Kgen}.  
Therefore Proposition \ref{variety} applies saying that the moduli stacks 
$M^{\alpha}(S_{\leq i};C_{\cdot})$ and $M^{\alpha}(S_{>i};C_{\cdot})$ exist as quasiprojective
varieties. 

The projective frame bundles over a basepoint
of $\rho _i$ are principal $PGL_2$-bundles denoted 
$$
F_{\leq i}\rightarrow M^{\alpha} (S_{\leq i};C_{\cdot})
$$
and
$$
F_{>i}\rightarrow M^{\alpha} (S_{>i};C_{\cdot}).
$$
These principal bundles may be viewed as given by the maps 
$$
M^{\alpha} (S_{\leq i};C_{\cdot}) \rightarrow M^{\alpha} (\rho _i)=
B(PGL_2) \leftarrow M^{\alpha} (S_{>i};C_{\cdot}).
$$
These principal bundles are locally trivial in the Zariski topology
by Lemma \ref{lem62}. 

The principal bundle
description of the moduli space in Proposition \ref{principal} now says
$$
M^{\alpha} (S;C_{\cdot}) = 
{\rm Iso} (p_1^{\ast} (F_{\leq i}), p_1^{\ast} (F_{\leq i}) ) / 
M^{\alpha} (S_{\leq i};C_{\cdot})\times M^{\alpha} (S_{>i};C_{\cdot}) .
$$
The bundle of isomorphisms between our two principal bundles, is a fiber bundle with
fiber $PGL_2$, locally trivial in the Zariski topology because the two principal bundles are
Zariski-locally trivial. We may sum up this conclusion with the following lemma, noting that
the argument also works the same way if $G_i = \{ -1\}$. 

\begin{lemma}
\label{idcase}
Under the assumption that $G_i = \{ 1\}$, the moduli space $M^{\alpha} (S;C_{\cdot})$ is a 
fiber bundle over $M(S_{\leq i};C_{\cdot})\times M(S_{>i};C_{\cdot})$, 
locally trivial in the Zariski topology,
with fiber $PGL_2$. The same holds true if $G_i = \{ -1\}$.
\end{lemma}

Consider the next case: suppose that $G_i$ is the conjugacy class of matrices
conjugate to a nontrivial unipotent
matrix
$$
\Uu = 
\left( \begin{array}{cc}
1 & 1 \\
0 & 1
\end{array}
\right) . 
$$
In that case, $M^{\alpha} (\rho _i)= B\gggg _a$. The situation is the same as before: the
moduli spaces $M^{\alpha} (S_{\leq i};C_{\cdot})$ 
and $M^{\alpha} (S_{>i};C_{\cdot}) $ are quasiprojective varieties, and we have
principal bundles $F_{\leq i}$ and $F_{>i}$. This time, these principal bundles have
unipotent automorphisms denoted $R'$ and $R$ respectively, in the conjugacy class of $\Uu$.
We have 
$$
M^{\alpha} (S;C_{\cdot})=
{\rm Iso} 
_{M^{\alpha} (S_{\leq i};C_{\cdot})\times M^{\alpha} (S_{>i};C_{\cdot})}
(p_1^{\ast} (F_{\leq i}, R'), p_1^{\ast} (F_{\leq i},R) )  .
$$
This means the relative isomorphism bundle of the principal bundles together with their 
automorphisms. 

We claim that these principal bundles together with their automorphisms may be trivialized locally
in the Zariski topology. For the principal bundles themselves this is Lemma \ref{lem62}. The unipotent
endomorphisms then correspond, with respect to these local trivializations, to maps
into $PGL_2/\gggg_a$. One can write down explicit sections of the projection
$PGL_2\rightarrow PGL_2/\gggg_a$ locally in the Zariski topology of the base, and these give the
claimed local trivializations. One might alternatively notice here that a $\gggg_a$-torsor
for the etale topology is automatically locally trivial in the Zariski topology by 
``Hibert's theorem 90''. 

From the result of the previous paragraph, $M^{\alpha} (S;C_{\cdot})$ is a fiber bundle over 
$M^{\alpha} (S_{\leq i};C_{\cdot})\times 
M^{\alpha} (S_{>i};C_{\cdot})$, locally trivial in the Zariski topology, with fiber the
centralizer $Z(R)\subset PGL_2$ of a unipotent element $R\in PGL_2$. 
This centralizer is $\gggg _a\cong \aaa^1$.
We obtain the following statement.

\begin{lemma}
\label{unicase}
Under the assumption that $G_i$ is the unipotent conjugacy class, the moduli space 
$M^{\alpha}(S;C_{\cdot})$ is a 
fiber bundle over $M^{\alpha} (S_{\leq i};C_{\cdot})\times 
M^{\alpha} (S_{>i};C_{\cdot})$, locally trivial in the Zariski topology,
with fiber $\aaa^1$. The same holds true if $G_i$ is the conjugacy class of matrices
conjugate to $-\Uu$. 
\end{lemma}

We may sum up the conclusion of this section as follows.

\begin{proposition}
With the hypothesis of Condition \ref{verygen} in effect,
suppose that the datum $\alpha$ is chosen such that for some $i$,
$G_i$ is one of the following four conjugacy classes
$$
\{ 1\} , \;\; 
\{ -1\} , \;\; 
\{ P\Uu P^{-1}\} , \mbox{ or }
\{ - P \Uu P^{-1} \} ,
$$
that is to say the conjugacy classes whose traces are $2$ or $-2$. 
Then the dual boundary complex of the $\alpha$-stratum is contractible: 
$$
\dualdel M^{\alpha}(S, C_{\cdot}) \sim \ast . 
$$
\end{proposition}
\begin{proof}
In all four cases, covered by Lemmas \ref{idcase} and \ref{unicase} above, the space 
$M^{\alpha}(S, C_{\cdot})$ admits a further decomposition into locally closed
pieces all of which have the form $\aaa^1\times Y$. Therefore, Lemmas \ref{idcase} and \ref{unicase} 
apply to show that the dual boundary complex is contractible. 
\end{proof}

\section{Decomposition at $S_i$ in the unstable case}
\label{sec-decomp-unstable}

Define the function $t_i : M(S,C_{\cdot})\rightarrow \aaa^1$ sending a local system to 
the trace of its monodromy around the circle $\rho _i$. In the previous section, we have treated
any strata which might be defined in such a way that at least one of the $G_i$ is a conjugacy
class with $t_i$ equal to $2$ or $-2$. Therefore, we may now assume that all of our subsets $G_i$
consist entirely of matrices with trace different from $2,-2$. In particular, these matrices are
semisimple with distinct eigenvalues. 

If $G_i$ consists of a single conjugacy class, it is 
possible to choose one of the two eigenvalues. But in general, this is not possible. However, in the
situation considered in the present section, where one of the $\sigma _i$ indicates an unstable
local system, then the destabilizing subsystem serves to  pick out a choice of eigenvalue. 

In the case where one of the $\sigma _i$ is $\{ 0\}$ stating that $V|_{S_i}$ should be unstable,
we will again obtain a structure of decomposition into a product with $\aaa^1$ locally over a
stratification, essentially by considering the extension class of the unstable local system. 
Some arguments are needed in order to show that this leads to direct product decompositions. 

\subsection{Some cases with $G_{i-1}$ and $G_i$ fixed}
\label{sec-fixed}

We suppose in this subsection that $G_{i-1}$ and $G_i$ are single conjugacy classes,
with traces different from $2,-2$, and furthermore
chosen so that the moduli problem for $M^{\alpha} (S_{> i};C_{\cdot})$ on one side
is Kostov-generic. Hence, that moduli stack is a quasiprojective variety. Furthermore we assume that
$\sigma _i = \{ 0\}$. Therefore, $M^{\alpha} (S_i;C_i)$ is the moduli stack of 
unstable local systems on $S_i$. 
The elements here are local systems $V$ fitting into an exact sequence
$$
0\rightarrow L \rightarrow V \rightarrow L' \rightarrow 0
$$
such that the monodromy of $L$ on $\xi _i$ has eigenvalue $c_i^{-1}$. 
We assume that $M^{\alpha}(S_i;C_i)$ is nonempty.

\begin{remark}
If we are given the conjugacy classes $G_{i-1}$ and $G_i$ such that there
exists an unstable local system $V$ on $S_i$, then the eigenvalues
$b_{i-1}$ of $L$ on $\rho _{i-1}$, and $b_i$ of $L$ on $\rho_i$, are 
uniquely determined.
\end{remark}
\begin{proof}
The conjugacy classes $G_{i-1}$, $G_i$ determine the pairs $(b_{i-1}, b_{i-1}^{-1})$
and  $(b_{i}, b_{i}^{-1})$ respectively. The instability condition says that $L$ has
eigenvalue $c_i^{-1}$ along $\xi_i$. Suppose that $b_{i-1}c_i^{-1}b_i=1$ so there
exists a local system $L$ with eigenvalues $b_{i-1}$ and $b_i$. We show that
the other products with either $b_{i-1}^{-1}$ or $b_i^{-1}$ or both, are different from $1$.
For example, $b_{i-1}c_i^{-1}b_i^{-1} = b_i^{-2}$, but $b_i^2\neq 1$ since we are assuming
that $G_i$ is a conjugacy class with distinct eigenvalues. Thus $b_{i-1}c_i^{-1}b_i^{-1}\neq 1$.
Similarly, $b_{i-1}^{-1}c_i^{-1}b_i\neq 1$. Also, $b_{i-1}^{-1}c_i^{-1}b_i^{-1}=c_i^{-2}\neq 1$. 
This shows that if there is one possible combination of eigenvalues for a sub-local system,
then it is unique. 
\end{proof}

From the assumption that $M^{\alpha} (S_i; C_i)$ is nonempty and the previous remark, we may
denote by $b_{i-1}$ and $b_i$ the eigenvalues of $L$ on $\rho _{i-1}$ and $\rho _i$
respectively.  

We are assuming a genericity condition implying 
that $M^{\alpha} (S_{> i}; C_{\cdot})$ is a quasiprojective variety. It has a universal
principal bundle $F_{>i}$ over it, and this has an automorphism $R$ corresponding
to the monodromy transformation around $\rho _i$. The eigenvalues of $R$ are $b_i$ and $b_i^{-1}$. 

Restrict to a finer stratification of $M^{\alpha} (S_{> i}; C_{\cdot})$ 
into strata denoted $M^{\alpha} (S_{> i})^a$ on
which $(F_{>i}, R)$ is trivial. Let $M^{\alpha} (S; C_{\cdot})^a$ be the 
inverse image of $M^{\alpha} (S_{>i}; C_{\cdot})^a$ under the map 
$M^{\alpha} (S; C_{\cdot})\rightarrow M^{\alpha} (S_{>i}; C_{\cdot})$. 

\begin{proposition}
We have 
$$
M^{\alpha} (S; C_{\cdot})^a = M^{\alpha} (S_{>i}; C_{\cdot})^a 
\times M^{\alpha} (S_{\leq i}; C_{\cdot})^{{\rm fr},R}
$$
where
$M^{\alpha} (S_{\leq i}; C_{\cdot})^{{\rm fr},R}$ is the moduli space of {\em framed} local systems,
that is to say local systems with a projective framing along $\rho _i$ compatible
with the monodromy and having the specified eigenvalues $(b_i,b_i^{-1})$. 
\end{proposition}
\begin{proof}
Use Proposition \ref{principal}. 
\end{proof}

Without the conditions $\alpha = (\sigma _{\cdot}, G_{\cdot})$, 
the framed moduli space is just the space of
sequences of group elements $A_1,\ldots , A_i$,
in conjugacy classes $C_1,\ldots , C_i$ respectively, 
such that $A_1\cdots A_iR=1$. 
Denote this space by 
$$
{\rm Rep} (C_1,\ldots , C_i; R).
$$ 
The moduli space $M^{\alpha}(S_{\leq i},C_{\cdot})^{{\rm fr},R}$ is the subspace of 
${\rm Rep} (C_1,\ldots , C_i; R)$ given by the conditions $\sigma _{\cdot}$ and $G_{\cdot}$. 

Notice here that, since we don't know a genericity condition for 
$(C_1,\ldots, C_i, G_i)$ the moduli space might not be smooth. Even though we are
considering framed representations, at a reducible representation the space
will in general have a singularity. Furthermore, the conditions $G_j$ might,
in principle, introduce 
other singularities. 

\begin{theorem}
\label{abelian}
With the above notations, 
let $R'$ be an element in the conjugacy class $G_{i-1}$. 
We have 
$$
M^{\alpha} (S_{\leq i}; C_{\cdot})^{{\rm fr},R}
\cong \aaa ^1 \times M^{\alpha} (S_{\leq i-1}; C_{\cdot})^{{\rm fr},R'}.
$$
\end{theorem}
\begin{proof}
It isn't too hard to see that the moduli space is an $\aaa^1$-bundle over
the second term on the right hand side, where the $\aaa^1$-coordinate is the extension
class. The statement that we would like to show, 
saying that there is a natural decomposition as a direct product, is a 
sort of commutativity property. 

Let ${\rm Rep} (C_1,\ldots , C_i; R)^u$ denote the subspace of
${\rm Rep} (C_1,\ldots , C_i; R)$ consisting of representations which are unstable on $S_i$.
This is equivalent to saying that $A_i$ fixes, and acts by $c_i^{-1}$ on 
the eigenvector of $R$ of eigenvalue $b_i$.
We will show an isomorphism 
$$
{\rm Rep} (C_1,\ldots , C_i; R)^u\cong \aaa^1\times {\rm Rep} (C_1,\ldots , C_{i-1}; R'),
$$
and this isomorphism will preserve the conditions $(\sigma _{\cdot}, G_{\cdot})$ over
$S_{i-1}$ so it restricts to an isomorphism between the moduli spaces as claimed in the
theorem. 

Write
$$
R = 
\left( \begin{array}{cc}
b_i^{-1} & 0 \\
0 & b_i
\end{array}
\right) .
$$
Then 
${\rm Rep} (C_1,\ldots , C_i; R)^u$ is the space of sequences $(A_1,\ldots , A_i)$
such that 
$$
A_1\cdots A_iR=1
$$
and 
\begin{equation}
\label{aidef}
A_i = 
\left( \begin{array}{cc}
c_i & 0 \\
y & c_i^{-1}
\end{array}
\right)
\end{equation}
for some $y\in \aaa^1$. 

Similarly, write 
$$
R' = 
\left( \begin{array}{cc}
b_{i-1}^{-1} & 0 \\
0 & b_{i-1}
\end{array}
\right) ,
$$
and ${\rm Rep} (C_1,\ldots , C_{i-1}; R')$ is the space of sequences $(A'_1,\ldots , A'_{i-1})$
such that 
$$
A'_1\cdots A'_{i-1} R' =1.
$$

Suppose $(A_1,\ldots , A_i)$ is a point in ${\rm Rep} (C_1,\ldots , C_i; R)^u$
and let $y\in \aaa^1$ be the lower left coefficient of $A_i$ from \eqref{aidef}. 
Note that $c_i^{-1}b_i =b_{i-1}$ so
$$
A_i R = \left( \begin{array}{cc}
b_i^{-1}c_i & 0 \\
b_i^{-1}y & c_i^{-1}b_i
\end{array}
\right)
=
\left( \begin{array}{cc}
b_{i-1}^{-1} & 0 \\
b_i^{-1}y & b_{i-1}
\end{array}
\right) .
$$
Let
$$
U:= \left( \begin{array}{cc}
1 & 0 \\
u & 1
\end{array}
\right)
$$
be chosen so that $UA_i R U^{-1} = R'$, which happens if and only if
$$
b_{i-1}^{-1} u + b_i ^{-1} y - b_{i-1} u = 0,
$$
in  other words
$$
u:= \frac{-b_i^{-1}y}{b_{i-1}^{-1}-b_{i-1}}.
$$
The denominator is nonzero because we are assuming the trace of $G_{i-1}$ is different
from $2$ or $-2$, which is equivalent to asking $b_{i-1}\neq b_{i-1}^{-1}$. 

Then put $A'_j:= UA_jU^{-1}$. From the equation $UA_i R U^{-1} = R'$ we get
$$
A'_1\cdots A'_{i-1}R' = U (A_1\cdots A_{i-1})U^{-1}(UA_i R U^{-1}) = 1.
$$
Hence, $(y, (A'_1,\ldots , A'_{i-1}))$ is a point in 
$\aaa^1\times {\rm Rep} (C_1,\ldots , C_{i-1}; R')$. 
This defines the map 
$$
{\rm Rep} (C_1,\ldots , C_i; R)^u\rightarrow \aaa^1\times {\rm Rep} (C_1,\ldots , C_{i-1}; R'),
$$
Its inverse is obtained by mapping $(y, (A'_1,\ldots , A'_{i-1}))$
to $(A_1,\ldots , A_i)$ where for $1\leq j \leq i-1$ we put 
$A_j = U^{-1} A'_j U$ with $U$ defined as above using $y$,
and $A_i$ is the upper triangular matrix \eqref{aidef}. We obtain the claimed isomorphism.
\end{proof}

By symmetry the same holds in case of Kostov-genericity on the other side, giving
a statement written as
$$
M^{\alpha} (S_{\geq i-1}; C_{\cdot})^{{\rm fr},R'}
\cong \aaa ^1 \times M^{\alpha} (S_{\geq i}; C_{\cdot})^{{\rm fr},R}.
$$

\subsection{Open $G_{i-1}$ and $G_i$}
\label{sec-variable}

If $\sigma _i=0$ and the moduli space is nonempty, then we cannot have both sides
being Kostov-nongeneric at once. 

Therefore, the remaining case is when $G_{i-1}$ and
$G_i$ are open sets which are 
unions of all but finitely many conjugacy classes (that is to say, allowing all traces but a finite number),
such that the moduli problems on both $S_{<i}$ and $S_{>i}$ are Kostov-generic.
In this situation, which we now assume,
the moduli spaces $M^{\alpha} (S_{<i}; C_{\cdot})$ and 
$M^{\alpha} (S_{>i}; C_{\cdot})$ exist and have principal bundles
$F_{<i}$ and $F_{>i}$ respectively. 

We have a map 
$$
M^{\alpha} (S_{<i}; C_{\cdot})\times M^{\alpha} (S_{>i}; C_{\cdot})
\rightarrow 
G_{i-1}\times G_i.
$$
Consider the etale covering space $\widetilde{G_i}$ 
which parametrizes matrices with a choice of one of the two eigenspaces.
This was considered extensively by Kabaya \cite{Kabaya}. 
Let 
$$
\widetilde{M}^{\alpha} (S_{>i}; C_{\cdot}):= 
M^{\alpha} (S_{>i}; C_{\cdot})\times _{G_i}\widetilde{G_i}
$$
and similarly for $\widetilde{M}^{\alpha} (S_{<i}; C_{\cdot})$.

Our hypothesis that $\sigma _i = \{ 0\}$, in other words that for any
local system $V$ in $M^{\alpha}(S;C_{\cdot})$ the restriction is unstable,
provides a factorization of the projection map through
$$
M^{\alpha}(S;C_{\cdot})\rightarrow 
\widetilde{M}^{\alpha} (S_{<i}; C_{\cdot})\times 
\widetilde{M}^{\alpha} (S_{>i}; C_{\cdot}).
$$
Indeed the destabilizing rank one subsystem is uniquely determined by the condition that
the monodromy around $\xi_i$ have eigenvalue $c_i^{-1}$, and this rank one subsystem
serves to pick out the eigenvalues of the matrices for $\rho _{i-1}$ and $\rho _i$. 

Now the same argument as before goes through. 
We may choose a stratification such that on each stratum the principal bundles have
framings such that the automorphisms $R'$ and $R$ are diagonal (note, however, that the
eigenvalues are now variable). 

We reduce to the following situation: $Z$ is a quasiprojective variety with
invertible functions $b_{i-1}$ and $b_i$ such that $b_{i_1}^{-1}c_i^{-1}b_i=1$, 
and we look at the moduli space of quadruples
$(z,V_i,\beta ',\beta )$ such that $z\in Z$, $V_i$ is an unstable local system on $S_i$, 
and 
$$
\beta ': V|_{\rho _{i-1}} \cong (V,R'(z)),
$$
$$
\beta : V|_{\rho _{i}} \cong (V,R(z))
$$
where 
$$
R'(z)= \left( \begin{array}{cc}
b_{i-1}(z)^{-1} & 0 \\
0 & b_{i-1}(z)
\end{array}
\right)
\;\;\;\;
R(z)= \left( \begin{array}{cc}
b_{i}(z)^{-1} & 0 \\
0 & b_{i}(z)
\end{array}
\right) .
$$
The map $Y=\beta ' \beta ^{-1}$ is an automorphism of $V$ (defined up to scalars, so it is a group element in $PGL_2$) and it preserves the marked subspace, so it is a lower-triangular matrix. It uniquely
determines the data $(V_i, \beta , \beta ')$ up to isomorphism. Indeed we may consider 
$V_i\cong V$ using for example $\beta '$, then our local system is 
$(R', A_i, YRY^{-1})$ where $A_i$ is specified by the condition $(R')^{-1}A_iYRY^{-1}=1$. 
As the group of lower triangular matrices in $PGL_2$ is isomorphic to $\Gm \times \gggg _a$
we obtain an isomorphism between our stratum and $Z\times \Gm \times \gggg _a$. 

Alternatively, one could just do a parametrized version of the proof of Theorem \ref{abelian}.

\subsection{Synthesis}
\label{sec-synth}

We may gather together the various cases that have been treated in this section so far. 

\begin{theorem}
Suppose $\alpha$ is any datum such that for some $i$ we have $\sigma _i = \{ 0\}$.
If $M^{\alpha}(S; C_{\cdot})$ is nonempty, then 
$\dualdel M^{\alpha}(S; C_{\cdot})\sim \ast$. 
\end{theorem}
\begin{proof}
Let $G^v$ be the set of matrices $R$ with ${\rm Tr}(R)\neq 2,-2$. In the previous section
we have treated the cases where any $G_i$ is one of the four conjugacy classes of trace $2$ or $-2$.
Therefore we may assume that $G_{i-1}, G_i\subset G^v$. 

Suppose that $G_{i-1}$ and $G_i$ are 
conjugacy classes chosen so that the sequences $(C_1,\ldots , C_{i-1}, G_{i-1})$ 
and $(G_i, C_{i+1}, \ldots , C_k)$ are both Kostov-nongeneric. 
Under the hypothesis $\sigma _i=\{ 0\}$ and supposing $M^{\alpha}(S; C_{\cdot})$ nonempty,
containing say a local system $V$,
then an eigenvalue of $G_{i-1}$ is the product of an eigenvalue of $C_i$ and an eigenvalue
of $G_i$, since there exists a rank one subsystem of $V|_{S_i}$. The same holds for the other
eigenvalue of $G_{i-1}$. Combining with the nongenericity relations among eigenvalues of
$(C_1,\ldots , C_{i-1}, G_{i-1})$ 
and $(G_i, C_{i+1}, \ldots , C_k)$, we obtain a nongenericity relation for $(C_1,\ldots , C_k)$. 
This contradicts the hypothesis of Condition \ref{Kgen} for $C_{\cdot}$. Therefore, we conclude that
if $M^{\alpha}(S; C_{\cdot})$ is nonempty, then for any specific choice of conjugacy classes
$G_{i-1}$ and $G_i$, at least one of the  moduli problems over $S_{<i}$ or $S_{>i}$ 
has to satisfy Condition \ref{Kgen}.
These cases are then covered by Theorem \ref{abelian} above. 

There are finitely many choices of single conjugacy classes $G_{i-1}$ (resp. $G_i$) such that 
$(C_1,\ldots , C_{i-1}, G_{i-1})$ 
(resp. $(G_i, C_{i+1}, \ldots , C_k)$) is Kostov non-generic. We may therefore isolate these choices
and treat them by Theorem \ref{abelian}  according to the previous paragraph. Let now $G_{i-1}$ and $G_i$ be
the complement in $G^v$ of these nongeneric conjugacy classes. These are open subsets such that
for any conjugacy classes therein, the moduli problems on $S_{<i}$ and $S_{>i}$ satisfy
Condition \ref{Kgen}. The discussion of subsection \ref{sec-variable} now applies to give the
conclusion that this part of $M^{\alpha}(S, C_{\cdot})$ has contractible dual boundary complex. 
\end{proof}

\section{Reduction to $M'$}
\label{sec-red}

In this section, we put together the results of the previous sections to obtain a reduction 
to the main biggest open stratum. Recall from Condition \ref{verygen}
that we are assuming that $C_{\cdot}$ is very
generic.  

Let the datum $\alpha '$ consist of the following choices: 
for all $i$, $\sigma '_i=\{ 1\}$
and $G_i$ is the set $G^v$ of matrices with trace $\neq 2,-2$. 
Then we put
$$
M':= M^{\alpha '} (S, C_{\cdot}) .
$$
It is an open subset of $M(S,C_{\cdot})$ since stability, and the conditions on the traces, are open
conditions. 

\begin{theorem}
There exists a collection of data denoted $\alpha ^j$ such that 
$$
M(S, C_{\cdot}) = M' \sqcup \coprod _j M^{\alpha ^j} (S, C_{\cdot})
$$
is a stratification, i.e. a decomposition into locally closed subsets admitting a total order 
satisfying the closedness condition of \ref{decomp}. Furthermore, this admits a
further refinement into a stratification with $M'$ together with pieces denoted
$Z^{j,a}\subset M^{\alpha ^j} (S, C_{\cdot})$, 
such that all of the pieces $Z^{j,a}$ have the form 
$$
Z^{j,a} = Y^{j,a} \times \aaa ^1.
$$
\end{theorem}
\begin{proof}
Let $G^v$ be the set of matrices of trace $\neq 2,-2$ and let $G^u$ be the set of matrices
of trace $2$ or $-2$. Let $\alpha ^j$ run over the $2^{2k-3}$ choices of 
$(\sigma _2,\ldots , \sigma _{k-1}; G_2,\ldots , G_{k-2})$ with $\sigma _i$ either
$\{ 0\}$ or $\{ 1\}$, and $G_i$ either $G^u$ or $G^v$.  
The locally closed pieces $M^{\alpha ^j} (S, C_{\cdot})$
are disjoint and their union is $M(S, C_{\cdot})$. Furthermore, the set of indices is partially
ordered with the product order induced by saying that $\{ 0\} < \{ 1\}$ and $G^u < G^v$
and $j_1\leq j_2$ if each component of $\alpha ^{j_1}$ is $\leq$ the corresponding component of
$\alpha ^{j_2}$.  
If $J$ is a downward cone in this partial ordering then 
$\bigcup _{j\in J} M^{\alpha ^j} (S, C_{\cdot})$ is closed, because specialization
decreases the indices (stable specializes to unstable and $G^v$ specializes to $G^u$). 
Choosing a compatible total ordering
we obtain the required closedness property. 

The highest element in the partial ordering is the datum 
$\alpha '$ considered above, so $M'$ is the open stratum of the stratification. 

The discussion of the previous two sections allows us to further decompose all of the
other strata $M^{\alpha ^j} (S, C_{\cdot})$, 
in a way which again preserves the ordered closedness condition, into pieces 
of the form $Z^{j,a} = Y^{j,a} \times \aaa ^1$. 
\end{proof}

\begin{corollary}
\label{cor92}
The natural map $\dualdel M(S,C_{\cdot}) \rightarrow \dualdel M'$
is a homotopy equivalence.
\end{corollary}
\begin{proof}
Apply Proposition \ref{decomp} to the stratification given by the theorem.
Note that $M'$ is nonempty and the full moduli space
is irreducible so the other strata are subvarieties of strictly smaller dimension. 
\end{proof}

\section{Fenchel-Nielsen coordinates}
\label{sec-main}

We are now reduced to the main case 
$M'= M^{\alpha}(S; C_{\cdot})$ for $\alpha '$ such that all $\sigma _i = \{ 1\}$ and all $G_i=G^v$. 
We would like to get
an expression for $M'$ allowing us to understand its dual boundary complex 
by inspection. We will show $M'\cong {\bf Q}^{k-3}$ where ${\bf Q}$ is defined near the
end of this section,
such that $\dualdel ({\bf Q}) \sim S^1$. The conclusion $\dualdel M' \sim S^{2(k-3)-1}$
then follows from Lemma \ref{join}. 

This product decomposition is a system of Fenchel-Nielsen coordinates for the
open subset $M'$ of the moduli space. 

One of the main things we learn from the basic theory of the classical hypergeometric function 
is that a rank two local system 
on $\pp^1 - \{ 0,1,\infty\}$ is heuristically determined by the three conjugacy classes of the 
monodromy transformations at the punctures. This general principle is not actually true, in cases
where there might be a reducible local system. But, imposing the condition of stability
provides a context in which this rigidity holds precisely. This is the statement of 
Corollary \ref{hyperge} below. 

Let $t_{i-1}$ and $t_i$ be points in $\aaa ^1- \{ 2,-2\}$. We will write down a stable
local system $V_i(t_{i-1}, t_i)$ on $S_i$, whose monodromy traces around $\rho _{i-1}$ and $\rho _i$
are $t_{i-1}$ and $t_i$ respectively, and whose monodromy around $\xi _i$ is in the
conjugacy class $C_i$. Furthermore, any stable local system with these traces 
is isomorphic to $V_i(t_{i-1}, t_i)$ in a unique way up to scalars. 

Construct $V_i(t_{i-1}, t_i)$ together with a basis at the basepoint $x_i$, by exhibiting monodromy
matrices $R'_{i-1}$, $R_i$ and $A_i$ in $SL_2$.  Set
$$
A_i := \left( \begin{array}{cc}
c_i & 0 \\
0 & c_i^{-1}
\end{array}
\right) 
\mbox{  and  }
R_i := \left( \begin{array}{cc}
u_i & 1 \\
w_i & (t_i-u_i)
\end{array}
\right) 
$$
with $u_i$ given by the formula \eqref{uiform} to be determined below, and $w_i:= u_i(t_i-u_i)-1$
because of the determinant one condition. 
 
We could just write down the formula for $u_i$ but in order to motivate it let us first calculate
$$
R'_{i-1}= A_iR_i = 
\left( \begin{array}{cc}
c_iu_i & c_i \\
c_i^{-1}w_i & c_i^{-1}(t_i-u_i)
\end{array}
\right) .
$$
We need to choose $u_i$ such that 
$$
t_{i-1} = {\rm Tr}(R'_{i-1}) = c_i u_i + c_i^{-1}(t_i-u_i).
$$
This gives the formula
\begin{equation}
\label{uiform}
u_i= \frac{t_{i-1}-c_i^{-1}t_i}{c_i-c_i^{-1}} .
\end{equation}
The denominator is nonzero since by hypothesis $c_i\neq c_i^{-1}$. 

\begin{lemma}
Suppose $V_i$ is an $SL_2$ local system with traces $t_{i-1}$ and $t_i$.
Suppose $V_i$ is given a frame at the base point $x_i$,
such that the monodromy matrix around the loop $\gamma _i$ is
diagonal with $c_i$ in the upper left, and such that the monodromy matrix
around $\rho _i$ (via the path going from $x_i$ to $s_i\in \rho _i$) 
has a $1$ in the upper right corner. Then the three monodromy
matrices of $V_i$ are the matrices $R'_{i-1}$, $R_i$ and $A_i$ defined above. 
\end{lemma}
\begin{proof}
The matrix $A_i$ is as given, by hypothesis. The matrix $R_i$ has trace $t_i$ 
and upper right entry $1$ by
hypothesis, so it too has to look as given. Now the calculation of the trace $t_{i-1}$
as a function of $u_i$ has a unique inversion: the value of $u_i$ must be given by
\eqref{uiform} as a function of $t_{i-1}$, $t_i$ and $c_i$. This determines the matrices.
\end{proof}

\begin{lemma}
Suppose $V_i$ is an $SL_2$ local system with traces $t_{i-1}$ and $t_i$
different from $2$ or $-2$, and
suppose $V_i$ is stable. Then, up to a scalar multiple, 
there is a unique frame for $V_i$ over the basepoint $x_i$ satisfying the conditions of 
the previous lemma.
\end{lemma}
\begin{proof}
Let $e_1$ and $e_2$ be eigenvectors for the monodromy around $\gamma _i$, with eigenvalues
$c_i$ and $c_i^{-1}$ respectively. They are uniquely determined up to a separate scalar for
each one. We claim that the upper right entry of the monodromy around $\rho _i$ is nonzero.
If it were zero, then the subspace generated by $e_2$ would be fixed, with the monodromy
around $\xi _i$ being $c_i^{-1}$; that would contradict the assumption of stability. 

Now since the upper right entry of the monodromy around $\rho _i$ is nonzero, we may
adjust the vectors $e_1$ and $e_2$ by scalars such that this entry is equal to $1$.
Once that condition is imposed, the only further allowable change of basis vectors 
is by multiplying $e_1$ and $e_2$ by the same scalar. 
\end{proof}

\begin{corollary}
\label{hyperge}
Suppose $V_i$ is a local system on $S_i$, with conjugacy class $C_i$ around $\xi _i$,
stable, and whose traces around $\rho _{i-1}$ and $\rho _i$ are
$t_{i-1}$ and $t_i$ respectively. Then there is up to a scalar an unique isomorphism  
$V_i\cong V_i(t_{i-1},t_i)$ with the system constructed above. 
\end{corollary}

Suppose $V$ is a point in $M'$, and let $t_i$ denote the traces of the monodromies
of $V$ around the loops $\rho_i$. Then by the definition of the datum $\alpha '$,
$t_i\neq 2,-2$ and
the restriction to each $S_i$ is stable, so by the corollary there
is up to scalars a unique isomorphism $h _i: V|_{S_i} \cong V_i(t_{i-1},t_i)$. 

Recall that $x_i$ is a basepoint in $S_i$, and that we have chosen a path in $S_i$
from $x_i$ to a basepoint $s_i$ in $\rho_i$, and then a path in $S_{i+1}$ from $s_i$ to $x_{i+1}$. 
Let $\psi _i$ denote composed the path from $x_i$ to $x_{i+1}$, and use the same symbol to denote
the transport along this path which is an isomorphism $\psi _i:V_{x_i}\cong V_{x_{i+1}}$. 
The stalk of the local system $V_i(t_{i-1},t_i)$ at $x_i$ is by construction $\cc^2$,
and the same at $x_{i+1}$, so the map 
$$
P_i:= h_{i+1}\psi _i h_i^{-1} : V_i(t_{i-1},t_i)_{x_i} \rightarrow V_{i+1}(t_{i},t_{i+1})_{x_{i+1}}
$$
is a matrix $P_i: \cc^2 \rightarrow \cc^2$ well-defined up to scalars, that is $P_i\in PGL_2$.

By the factorization property of $M'$, the local system $V$ is
determined by these glueing isomorphisms $P_i$, subject to the constraint that 
they should intertwine the monodromies around the circle $\rho_i$
for $V_i$ and $V_{i+1}$. We have used the notation $R'_i$ for the monodromy 
of the local system $V_{i+1}$ around the circle $\rho _i$, whereas 
$R_i$ denotes the monodromy of $V_i$ around here. We will have made sure to use the same paths
from $x_i$ or $x_{i+1}$ to the basepoint $s_i\in \rho _i$ in order to define these monodromy matrices
as were combined together to make the path $\psi _i$. Therefore, the compatibility
condition for $P_i$ says 
\begin{equation}
\label{Pcond}
R'_i \circ P_i = P_i \circ R_i .
\end{equation}
The frames for $V_{x_i}$ are only well-defined up to scalars, so the matrices $P_i$ are
only well-defined up to scalars and conversely if we change them by scalars then it doesn't change
the isomorphism class of the local system. 
Putting together all of these discussions, we obtain the following preliminary description of $M'$. 

\begin{lemma}
The moduli space $M'$ is isomorphic to the space of 
$(t_2,\ldots , t_{k-2})\in (\aaa ^1 - \{ 2,-2\} )^{k-3}$ and
$(P_2,\ldots , P_{k-2})\in (PGL_2)^{k-3}$ subject to the equations \eqref{Pcond},
where $R'_i$ and $R_i$ are given by the previous formulas in terms of the $t_j$.

For the end pieces, one should formally set 
$t_1:= c_1+c_1^{-1}$ and $t_{k-1}:= c_k + c_k^{-1}$.
\end{lemma}

At this point, we have not yet obtained a good ``Fenchel-Nielsen'' style coordinate system,
because the equation \eqref{Pcond} for $P_i$ contains $R'_i$ which depends on 
$t_{i+1}$ as well as $t_i$, and $R_i$ which depends on $t_{i-1}$ as well as $t_i$. 

We now proceed to decouple the equations. The strategy is to introduce the matrices
$$
T_i := \left( \begin{array}{cc}
0 & 1 \\
-1 & t_i
\end{array}
\right) 
$$
which serve as a canonical normal form for matrices with given traces $t_i$, not
requiring the marking of one of the two eigenvalues. Notice that if we set
$$
U_i:= 
\left( \begin{array}{cc}
1 & 0 \\
u_i & 1
\end{array}
\right) 
$$
then 
$$
U_i^{-1}T_iU_i= 
\left( \begin{array}{cc}
1 & 0 \\
-u_i & 1
\end{array}
\right) 
\left( \begin{array}{cc}
0 & 1 \\
-1 & t_i
\end{array}
\right) 
\left( \begin{array}{cc}
1 & 0 \\
u_i & 1
\end{array}
\right) 
= 
\left( \begin{array}{cc}
u_i & 1 \\
w_i & t_i-u_i
\end{array}
\right) 
$$
with $w_i$ as before. Therefore, using the formula \eqref{uiform} for $u_i$ we may
write
$$
R_i(t_{i-1}, t_i) = U_i^{-1}T_iU_i.
$$
Now
$$
R'_{i-1}= A_i R_i = A_i U_i^{-1}T_iU_i = U_i^{-1} (U_i A_i U_i^{-1} T_i )U_i .
$$
Furthermore, $U_iA_i U_i^{-1} $ is lower triangular with $c_i$ and $c_i^{-1}$ along
the diagonal, and when we multiply with $T_i$ it gives a matrix of the form
$$
U_i^{-1} A_i U_iT_i = 
\left( \begin{array}{cc}
c_i & 0 \\
\ast & c_i^{-1}
\end{array}
\right) 
\left( \begin{array}{cc}
0 & 1 \\
-1 & t_i
\end{array}
\right) 
= 
\left( \begin{array}{cc}
0 & c_i \\
-c_i^{-1} & \ast
\end{array}
\right) .
$$
However, we know that $u_i$ was chosen so that this matrix has trace $t_{i-1}$
(it is conjugate to $R'_{i-1}$), therefore in fact
$$
U_i^{-1} A_i U_iT_i =
\left( \begin{array}{cc}
0 & c_i \\
-c_i^{-1} & t_{i-1}
\end{array}
\right)
$$
as could alternately be seen by direct computation. 
By inspection this matrix is conjugate to $T_{i-1}$ as it should be from its trace.
Interestingly enough, the conjugation is by the matrix
$$
A^{\frac{1}{2}}_i:= 
\left( \begin{array}{cc}
c_i^{\frac{1}{2}} & 0 \\
0 & c_i^{-\frac{1}{2}}
\end{array}
\right) ,
$$
with
$$
U_i^{-1} A_i U_iT_i = A^{\frac{1}{2}}_iT_{i-1} A^{-\frac{1}{2}}_i .
$$
This half-power seems also to occur 
somewhere in the classical treatments of the Fenchel-Nielsen coordinates,

We obtain
$$
R'_{i-1} = U_i^{-1} (U_i A_i U_i^{-1} T_i )U_i  = 
U_i^{-1} A^{\frac{1}{2}}_iT_{i-1} A^{-\frac{1}{2}}_i U_i .
$$
Recall that the equation \eqref{Pcond} for $P_{i-1}$ reads
$$
R'_{i-1} \circ P_{i-1} = P_{i-1} \circ R_{i-1},
$$
and using the above formula for  $R'_{i-1} $ as well
as $R_{i-1} = U_{i-1}^{-1}T_{i-1}U_{i-1}$, this equation reads
\begin{equation}
\label{intermediatecond}
U_i^{-1} A^{\frac{1}{2}}_iT_{i-1} A^{-\frac{1}{2}}_i U_i \circ P_{i-1}
=P_{i-1} \circ U_{i-1}^{-1}T_{i-1}U_{i-1}.
\end{equation}
Set
$$
Q_{i-1}:= A^{-\frac{1}{2}}_i U_i  P_{i-1} U_{i-1}^{-1}.
$$
This is a simple change of variables of the matrix $P_{i-1}$, with the matrices
entering into the change of variables depending however on $t_{i-2}$, $t_{i-1}$ and $t_i$. 
Notice that the coefficients of $Q_{i-1}$ are linear functions of the coefficients of 
$P_{i-1}$, in particular the action of scalars is the same on both. 

Our equation which was previously \eqref{Pcond} (but for $i-1$ instead of $i$), 
has become \eqref{intermediatecond}
which, after multiplying on the left by $U_i$ then by $A^{-\frac{1}{2}}_i$ and on the
right by $U_{i-1}^{-1}$ and substituting $Q_{i-1}$, 
becomes:
\begin{equation}
\label{Qcond}
T_{i-1}\circ Q_{i-1} = Q_{i-1} T_{i-1}.
\end{equation}
A sequence of matrices $Q_i$ satisfying these equations leads back to a sequence of matrices
$P_i$ satisfying \eqref{Pcond}  and vice-versa. Recall that the glueing for the local system
depended on these matrices modulo scalars, that is to say in $PGL_2$. We may sum up with the
following proposition:

\begin{proposition}
\label{feniprelim}
The moduli space $M'$ is isomorphic to the space of  
choices of 
$$
(t_2,\ldots , t_{k-2})\in (\aaa ^1 - \{ 2,-2\} )^{k-3}\mbox{  and  }
(Q_2,\ldots , Q_{k-2})\in (PSL_2)^{k-3}
$$
subject to the equations $T_iQ_i=Q_iT_i$. 
\end{proposition}

This expression for the  moduli space is now decoupled, furthermore the equations are
in a nice and simple form. 

\begin{theorem}
\label{fenchelnielsen}
Let ${\bf Q}$ be the space of pairs $(t,[p:q]) \in \aaa^1\times \pp^1$ such that
$t\neq 2,-2$ and 
\begin{equation}
\label{Qdef}
p^2+tpq + q^2 \neq 0.
\end{equation} 
Then we have
$$
M'\cong {\bf Q}^{k-3}.
$$
\end{theorem}
\begin{proof}
This will follow from the previous proposition, once we calculate that the
space of matrices $Q_i$ in $PGL_2$ commuting with $T_i$, is equal to the
space of points $[p,q]\in \pp^1$ such that $p^2+t_ipq + q^2 \neq 0$.
Write
$$
Q_i = \left( \begin{array}{cc}
p & q \\
p' & q'
\end{array}
\right) 
$$
then 
$$
Q_iT_i = 
\left( \begin{array}{cc}
p & q \\
p' & q'
\end{array}
\right)
\left( \begin{array}{cc}
0 & 1 \\
-1 & t_i
\end{array}
\right)
=
\left( \begin{array}{cc}
-q & p+t_iq \\
-q' & p' + t_i q'
\end{array}
\right)
$$
whereas
$$
T_iQ_i = 
\left( \begin{array}{cc}
0 & 1 \\
-1 & t_i
\end{array}
\right)
\left( \begin{array}{cc}
p & q \\
p' & q'
\end{array}
\right)
=
\left( \begin{array}{cc}
p' & q' \\
t_ip'-p & t_i q'-q
\end{array}
\right) .
$$
The equation $Q_iT_i=T_iQ_i$ thus gives from the top row
$$
p'=-q,\;\;\;\; q' = p+t_iq
$$
and then, those actually make the other two equations hold automatically. 
Therefore a solution $Q_i$ may be written
$$
Q_i = \left( \begin{array}{cc}
p & q \\
-q & p+t_iq
\end{array}
\right) .
$$
The statement $Q_i\in PGL_2$ means that $Q_i$ is taken up to multiplication by scalars,
in other words $[p:q]$ is a point in $\pp^1$ (clearly those coordinates are not both zero);
and 
$$
{\rm det}(Q_i) = p^2 + t_i pq + q^2 \neq 0.
$$
We conclude that the space of $(t_i,Q_i)\in (\aaa^1- \{ 2,-2\}) \times PGL_2$
such that $T_iQ_i=Q_iT_i$ is isomorphic to ${\bf Q}$. Therefore
Proposition \ref{feniprelim} now says $M'\cong {\bf Q}^{k-3}$. 
\end{proof}

\begin{lemma}
\label{ddQ}
The dual boundary complex of ${\bf Q}$ is 
$$
\dualdel {\bf Q} \sim S^1.
$$
Therefore 
$$
\dualdel {\bf Q}^{k-3} \sim S^{2(k-3)-1}.
$$
\end{lemma}
\begin{proof}
Let $\Phi \subset \pp^1\times \aaa^1$ be the open subset defined by the same inequation
\eqref{Qdef}. Then ${\bf Q}\subset \Phi$ is an open subset, whose complement is the disjoint union
of two affine lines. Furthermore, $\overline{\Phi}:= \pp^1\times \pp^1$ is a 
(non simple) normal crossings
compactification of $\Phi$. The divisor at infinity is the union of two copies of $\pp^1$,
namely the fiber over $t=\infty$ and the conic defined by $p^2+tpq+q^2=0$. These intersect
transversally in two points. Therefore, the incidence complex of $\Phi\subset \overline{\Phi}$
at infinity is a graph with
two vertices and two edges joining them. 

It follows that the incidence complex at infinity for ${\bf Q}$ is a circle. That may also be seen directly
by blowing up two times over each ramification point of the conic lying over $t=\pm 2$. 

Now applying Lemma \ref{join} successively, and noting that the successive join of $k-3$ times 
the circle is $S^{2(k-3)-1}$, we obtain the second statement. 
\end{proof}

\begin{corollary}
\label{cor-mainthm}
Let $C_{\cdot}$ be a collection of conjugacy classes satisfying Condition \ref{verygen}.
Then the moduli space 
${\rm M}_B(S; C_1,\ldots , C_k)$ of rank
$2$ local systems with those prescribed conjugacy classes, has dual boundary
complex homotopy equivalent to a sphere
$$
\dualdel {\rm M}_B(S; C_1,\ldots , C_k) \sim S^{2(k-3)-1}.
$$
\end{corollary}
\begin{proof}
We have been working with the hybrid moduli stack $M(S;C_{\cdot})$ above, but 
Proposition \ref{variety} says that this is the same as the moduli space 
${\rm M}_B(S; C_1,\ldots , C_k)$. By Corollary \ref{cor92}, 
$\dualdel M(S;C_{\cdot})\sim \dualdel M'$. By Theorem \ref{fenchelnielsen},
$M'\cong {\bf Q}^{k-3}$, and by Lemma \ref{ddQ} $\dualdel {\bf Q}^{k-3}\sim 
S^{2(k-3)-1}$. Putting these all together we obtain the desired conclusion. 
\end{proof}

This completes the proof of Theorem \ref{main}. 

\begin{remark}
The space $\Phi ^{k-3}$ itself has a modular interpretation: it is $M^{\alpha}(S,C_{\cdot})$
for $\alpha$ given by setting all $\sigma _i$ to $\{ 1\}$ (requiring stability of each $V|_{S_i}$),
but having $G_i=GL_2$ for all $i$, that is no longer constraining the traces. 
\end{remark}
 
\section{A geometric $P=W$ conjecture}
\label{rel-hitch}

In this section we discuss briefly the relationship between the theorem proven above,
and the Hitchin fibration. For this discussion, let us suppose that the eigenvalues
$c_i$ are $n_i$-th roots of unity, so the conjugacy classes $C_i$ have finite order $n_i$.
Fix points $y_1,\ldots , y_k\in \pp^1$ and 
let 
$$
X:= \pp^1[ \frac{1}{n_1}y_1, \ldots , \frac{1}{n_k}y_k]
$$
be the {\em root stack} with denominators $n_i$ at the points $y_i$ respectively. It is
a smooth proper Deligne-Mumford stack. The fundamental group of its topological
realization \cite{Noohi} is generated by the
paths $\gamma _1,\ldots , \gamma _k$ subject to the relations that $\gamma _1\cdots \gamma _k=1$ and
$\gamma _i^{n_i}=1$. We may also let $S$ be a punctured sphere such as considered above, 
the complement of a collection of small discs in
$\pp ^1$ centered at the points $y_i$. Therefore, a local system on $X^{\rm top}$ is the same thing
as a local system on $S$ such that the monodromies around the boundary loops $\xi _i$ have order $n_i$ respectively. We have
$$
{\rm M}_B(X^{\rm top}, GL_r) = \coprod _{(C_1,\ldots , C_k)} {\rm M}_B(S, C_{\cdot})
$$
where the disjoint union runs over the sequences of conjugacy classes such that $C_i$ has order
$n_i$. 
Recall that if we assume that $C_{\cdot}$ satisfies the Kostov-genericity condition
then the
character variety with fixed conjugacy classes ${\rm M}_B(S, C_{\cdot})$
is the same as the hybrid moduli stack $M(S, C_{\cdot})$.
It may be seen as a connected component of the character variety ${\rm M}_B(X^{\rm top},GL_r)$. 

Now we recall that there is a homeomorphism between the character variety 
${\rm M}_B(X^{\rm top},GL_r)$ and the Hitchin-Nitsure moduli space 
${\rm M}_{Dol}(X^{\rm top},GL_r)$ of Higgs bundles. One may consult for example
\cite{lspavm}, \cite{Konno}, \cite{Nakajima} for the general theory in the open or orbifold setting. 
We denote by ${\rm M}_{Dol}(S, C_{\cdot})$ the connected component of
${\rm M}_{Dol}(X^{\rm top},GL_r)$ corresponding to the choice of conjugacy classes, which it may be
recalled corresponds to fixing appropriate parabolic weights for the parabolic Higgs bundles. 
Hitchin's equations give a homeomorphism , the ``nonabelian Hodge correspondence''
\begin{equation}
\label{nahc}
{\rm M}_{Dol}(S, C_{\cdot})^{\rm top}\cong {\rm M}_B(S, C_{\cdot})^{\rm top}.
\end{equation}
Recall that the resulting two complex structures on the same underlying moduli space,
form a part of a hyperk\"ahler triple \cite{Hitchin}. 

In the smooth proper orbifold setting we have the same theory  of the Hitchin map 
$$
{\rm M}_{Dol}(S, C_{\cdot}) \rightarrow \aaa^n
$$
which is a Lagrangian fibration to the space of integrals of Hitchin's Hamiltonian system
\cite{HitchinDuke}. In particular, $n$ is one-half of the complex dimension of the moduli
space, that dimension being even because of the hyperk\"ahler structure. 

Fix a neighborhood of infinity in the Hitchin base ${\mathbb B}^{\ast}\subset \aaa ^n$,
and let $N^{\ast}_{Dol}$ denote its preimage in ${\rm M}_{Dol}(S, C_{\cdot})$. Similarly,
let $N^{\ast}_B$ denote a neighborhood of infinity in ${\rm M}_B(S, C_{\cdot})$.
The homeomorphism \ref{nahc} gives a natural homotopy equivalence
$N^{\ast}_{Dol}\sim N^{\ast}_B$. 

The neighborhood at infinity ${\mathbb B}^{\ast}\subset \aaa ^n$ has the homotopy
type of the sphere $S^{2n-1}$, and indeed we may view 
$S^{2n-1}$ as the quotient of ${\mathbb B}^{\ast}$ by radial scaling, 
so the Hitchin map provides a natural map 
$$
N^{\ast}_{Dol}\rightarrow S^{2n-1}.
$$

On the other hand, there is a natural projection $N^{\ast}_B\rightarrow \dualdel {\rm M}_B(S, C_{\cdot})$.
This is a general phenomenon, indeed if we have chosen a very simple normal crossings compactification
with divisor components $D_1,\ldots , D_m$ then we may choose an open covering of 
$N^{\ast}_B$ by open subsets $U_1,\ldots , U_m$ punctured neighborhoods of the $D_i$, such that
$U_{i_1}\cap \cdots \cap U_{i_r}$ is nonempty if and only if 
$D_{i_1}\cap \cdots \cap D_{i_r}$ is nonempty. Then, any partition of unity for this covering
provides a map $N^{\ast}_B\rightarrow \rr^m$ which just goes into the subspace 
$\dualdel {\rm M}_B(S, C_{\cdot})$. 

Recall the following conjecture \cite{KNPS}, which was motivated by consideration of
the case $\pp ^1-\{ y_1,y_2,y_3,y_4\}$.  

\begin{conjecture}
\label{geopw}
There is a homotopy-commutative square
$$
\begin{array}{ccc}
N^{\ast}_{Dol}& \stackrel{\sim}{\rightarrow} & N^{\ast}_B\\
\downarrow & & \downarrow \\
S^{2n-1}& \stackrel{\sim}{\rightarrow} & \dualdel {\rm M}_B(S, C_{\cdot})
\end{array}
$$
where the top and side maps are those described above, such that the bottom map is a
homotopy equivalence. 
\end{conjecture}

Our main theorem provides a homotopy equivalence such as the one which is conjectured to
exist on the bottom of the square, for the group $GL_2$ on $\pp^1-\{ y_1,\ldots , y_k\}$. 
This was our motivation, and it was also the motivation for Komyo's proof in the case $k=5$ 
\cite{Komyo}. 

We haven't shown anything about commutativity of the diagram. This is one of the motivations for
looking at the geometric theory of harmonic maps to buildings developed in \cite{KNPS} \cite{KNPS2}. 
A result in this direction is shown by Daskalopoulos, Dostoglou and Wentworth \cite{DDW}.
The Kontsevich-Soibelman wallcrossing picture 
\cite{KontsevichSoibelmanWallcrossing}
should provide a global framework for this question. 

Conjecture \ref{geopw} may be viewed as a geometrical analogue of the first weight-graded piece of the
$P=W$ conjecture 
\cite{deCataldoHauselMigliorini}
\cite{Hausel}. That conjecture states that weight filtration $W$
of the mixed Hodge structure on the
cohomology of the character variety $M_B$ should be naturally identified with the 
perverse Leray filtration $P$ induced by the Hitchin fibration. 
For the case of rank two character varieties on a compact Riemann surface, it 
was in fact proved by de Cataldo, Hausel and Migliorini \cite{deCataldoHauselMigliorini}. 
Davison treats a twisted version \cite{Davison}. 

It is known \cite{Payne} that the cohomology of the dual boundary
complex is the first weight-graded piece of the cohomology of $M_B$. Conjecture 
\ref{geopw} states that this should come from the cohomology of the sphere at infinity in the
Hitchin fibration, which looks very much like a Leray piece. 

Furthermore, indeed from discussions with L. Migliorini and S. Payne
it seems to be the case that the characterization of the cohomology of the dual boundary
complex in \cite{Payne}, and the computations 
\cite{HauselThaddeus1} \cite{HauselThaddeus2} \cite{HauselLetellierRodriguez}
\cite{HauselLetellierRodriguez2}
of the cohomology ring of 
${\rm M}_{Dol}$ used to prove the $P=W$ conjecture for $SL_2$
in \cite{deCataldoHauselMigliorini}, should serve to show commutativity of the diagram in rational
cohomology. 

The question of proving the analogue of our Theorem \ref{main} for a compact Riemann surface,
even in the rank $2$ case, is an interesting problem for further study. One may also envision
the case of a punctured curve of higher genus. The techniques used here involved a choice of
stability condition on each of the pieces of the decomposition, which in the higher genus
case would require having at least a certain number of punctures. Weitsman suggests,
following \cite{Weitsman} and \cite{JeffreyWeitsman}, 
that it might be possible to obtain a similar argument with only
at least one puncture. The compact case would seem to be more difficult to handle. 

Let us note that Kabaya \cite{Kabaya} gives a general discussion of coordinate systems
which can be obtained using decompositions, and he treats the problems of indeterminacy of
choices of eigenspaces up to permutations. 

The other direction which needs to be considered is local systems of higher rank.
Here, the first essential case is $\pp^1-\{ 0,1,\infty \}$, where there is no useful decomposition
of the surface into simpler pieces. We could hope that if this basic case could be treated in all ranks,
then the reduction techniques we have used above could allow for an extension to the case of
many punctures.

\bibliographystyle{amsplain}

\end{document}